\documentclass[a4paper]{scrartcl}
\usepackage[utf8]{inputenc}
\usepackage{graphicx}
\usepackage{varioref}
\usepackage{amsmath,amssymb}
\usepackage{amsthm}
\usepackage{color}
\usepackage{listings}
\usepackage[UKenglish]{babel}
\usepackage[square]{natbib}
\usepackage{bm}
\usepackage{eso-pic} 
\usepackage{hyperref} 

\DeclareMathOperator*{\argmin}{arg\,min\,}
\DeclareMathOperator{\sign}{sign}
\DeclareMathOperator{\diag}{diag}
\DeclareMathOperator{\Bias}{Bias}
\DeclareMathOperator{\Var}{Var}
\DeclareMathOperator{\MSE}{MSE}

\renewcommand{\mathbf}[1]{\ensuremath{\bm{\mathrm{#1}}}}

\normalsize
 
\theoremstyle{plain}
\newtheorem{theorem}{Theorem}[section]
\newtheorem{prop}[theorem]{Proposition}
\newtheorem{cor}[theorem]{Corollary}
\newtheorem{lemma}[theorem]{Lemma}

\begin{document}


\AddToShipoutPictureBG*{%
  \AtPageUpperLeft{%
    \raisebox{-\baselineskip}{%
      \makebox[5pt][l]{\quad \small This is an Author's Accepted Manuscript of an article published in \textit{Statistics: A Journal of Theoretical and Applied Statistics} 2014,}}
    \raisebox{-2\baselineskip}{
      \makebox[5pt][l]{\small available online: \url{http://www.tandfonline.com/doi/full/10.1080/02331888.2014.922563}}} }}%

\title{\Large The Influence Function of Penalized Regression Estimators}
\author{\normalsize Viktoria Öllerer  $^{a}$$^{\ast}$\thanks{$^\ast$Corresponding author. Email: viktoria.oellerer@kuleuven.be \vspace{6pt}}, Christophe Croux $^{a}$ and Andreas Alfons $^{b}$\\\normalsize\vspace{6pt} $^{a}${\em{Faculty of Economics and Business, KU Leuven, Belgium}};\\\normalsize $^{b}${\em{Erasmus School of Economics, Erasmus University Rotterdam, The Netherlands}}}
\date{\vspace{-0.5cm}}
\maketitle

\begin{abstract}
\noindent To perform regression analysis in high dimensions, lasso or ridge estimation are a common choice. However, it has been shown that these methods are not robust to outliers. Therefore, alternatives as penalized M-estimation or the sparse least trimmed squares (LTS) estimator have been proposed. The robustness of these regression methods can be measured with the influence function. It quantifies the effect of infinitesimal perturbations in the data. Furthermore it can be used to compute the asymptotic variance and the mean squared error. In this paper we compute the influence function, the asymptotic variance and the mean squared error for penalized M-estimators and the sparse LTS estimator. The asymptotic biasedness of the estimators make the calculations nonstandard. We show that only M-estimators with a loss function with a bounded derivative are robust against regression outliers. In particular, the lasso has an unbounded influence function.\\
\smallskip 

\smallskip

\end{abstract}

\section{Introduction}
\label{sec:Intro}

Consider the usual regression situation. We have data $(X, \mathbf{y})$, where $X\in\mathbb{R}^{n\times p}$ is the predictor matrix and $\mathbf{y}\in\mathbb{R}^n$ the response vector. A linear model is commonly fit using least squares  regression. It is well known that the least squares estimator suffers from large variance in presence of high multicollinearity among the predictors. To overcome these problems, ridge \citep[][]{Hoerl} and lasso estimation \citep[][]{Tibshirani} add a penalty term to the objective function of least squares regression
\begin{align}
\hat{\beta}_{LASSO}&=\argmin_{\mathbf{\beta}\in\mathbb{R}^p} \frac{1}{n}\sum_{i=1}^n (y_i-\mathbf{x}_i'\mathbf{\beta})^2 + 2\lambda\sum_{j=1}^p |\beta_j| \label{eq:lasso_data}\\
\hat{\beta}_{RIDGE}&=\argmin_{\mathbf{\beta}\in\mathbb{R}^p} \frac{1}{n}\sum_{i=1}^n (y_i-\mathbf{x}_i'\mathbf{\beta})^2 + 2\lambda\sum_{j=1}^p \beta_j^2.
\end{align}
In contrast to the ridge estimator that only shrinks the coefficients of the least squares estimate $\mathbf{\hat{\beta}}_{LS}$, the lasso estimator also sets many of the coefficients to zero. This increases interpretability, especially in high-dimensional models. The main drawback of the lasso is that it is not robust to outliers. As \cite{Alfons} have shown, the breakdown point of the lasso is 1/n. This means that only one single outlier can make the estimate completely unreliable. 

Hence, robust alternatives have been proposed. The least absolute deviation (LAD) estimator is well suited for heavy-tailed error distributions, but does not perform any variable selection. To simultaneously perform robust parameter estimation and variable selection, \cite{Wang2} combined LAD regression with lasso regression to LAD-lasso regression. However, this method has a finite sample breakdown point of $1/n$ \citep[][]{Alfons}, and is thus not robust. Therefore \cite{Arslan} provided a weighted version of the LAD-lasso that is made resistant to outliers by downweighting leverage points.

A popular robust estimator is the least trimmed squares (LTS) estimator \citep[][]{RousseeuwLeroy}. Although its simple definition and fast computation make it interesting for practical application, it cannot be computed for high-dimensional data ($p>n$). Combining the lasso estimator with the LTS estimator, \cite{Alfons} developed the sparse LTS-estimator 
\begin{align}
\mathbf{\hat{\beta}}_{spLTS} = \argmin_{\mathbf{\beta}\in\mathbb{R}^p} \frac{1}{h}\sum_{i=1}^h r_{(i)}^2(\mathbf{\beta}) + \lambda \sum_{j=1}^p |\beta_j|,
\label{eq:spLTS_sample}
\end{align}
where $r_i^2(\mathbf{\beta}) = (y_i - \mathbf{x}'_i\mathbf{\beta})^2$ denotes the squared residuals and $r_{(1)}^2(\mathbf{\beta})\leq\ldots\leq r_{(n)}^2(\mathbf{\beta})$ their order statistics. Here $\lambda\geq 0$ is a penalty parameter and $h\leq n$ the size of the subsample that is considered to consist of non-outlying observations. This estimator can be applied to high-dimensional data with good prediction performance and high robustness. It also has a high breakdown point \citep[][]{Alfons}.

All estimators mentioned until now, except the LTS and the sparse LTS-estimator, are a special case of a more general estimator, the penalized M-estimator \citep[][]{Li}
\begin{align}
\label{eq:betaM_data}
\mathbf{\hat{\beta}}_M=\argmin_{\mathbf{\beta}\in\mathbb{R}^p} \frac{1}{n}\sum_{i=1}^n \rho(y_i-\mathbf{x}_i'\mathbf{\beta}) + 2\lambda\sum_{j=1}^p J(\beta_j),
\end{align}
with loss function $\rho:\mathbb{R}\rightarrow\mathbb{R}$ and penalty function $J:\mathbb{R}\rightarrow\mathbb{R}$. While lasso and ridge have a quadratic loss function $\rho(z)=z^2$, LAD and LAD-lasso use the absolute value loss $\rho(z)=|z|$. The penalty of ridge is quadratic $J(z)=z^2$, whereas lasso and LAD-lasso use an $L_1$-penalty $J(z)=|z|$, and the `penalty' of least squares and LAD can be seen as the constant function $J(z)=0$. In the next sections we will see how the choice of the loss function affects the robustness of the estimator. In Equation (\ref{eq:betaM_data}), we implicitly assume that scale of the error term is fixed and known, in order to keep the calculations feasible. In practice, this implies that the argument of the $\rho$-function needs to be scaled by a preliminary scale estimate. Note that this assumption does not affect the lasso or ridge estimator.

The rest of the paper is organized as follows. In Section \ref{sec:Funct}, we define the penalized M-estimator at a functional level. In Section \ref{sec:Bias}, we study its bias for different penalties and loss functions. We also give an explicit solution for sparse LTS for simple regression. In Section \ref{sec:IF} we derive the influence function of the penalized M-estimator. Section \ref{sec:IF_lasso} is devoted to the lasso. We give its influence function and describe the lasso as a limit case of penalized M-estimators with a differentiable penalty function. For sparse LTS we give the corresponding influence function in Section \ref{sec:IF_spLTS}. In Section \ref{sec:Plot_IF} we compare the plots of influence functions varying loss functions and penalties. A comparison at sample level is provided in Section \ref{sec:NumExp}. Using the results of Sections \ref{sec:IF} - \ref{sec:IF_spLTS}, Section \ref{sec:ASVMSE} compares sparse LTS and different penalized M-estimators by looking at asymptotic variance and mean squared error. Section \ref{sec:Conc} concludes. The appendix contains all proofs.


\section{Functionals}
\label{sec:Funct}

Throughout the paper we work with the typical regression model 
\begin{align}
y=\mathbf{x}'\mathbf{\beta}_0+e
\label{eq:ModelF}
\end{align}
with centered and symmetrically distributed error term $e$. The number of predictor variables is $p$ and the variance of the error term $e$ is denoted by $\sigma^2$. We assume independence of the regressor $\mathbf{x}$ and the error term $e$ and denote the joint model distribution of $\mathbf{x}$ and $y$ by $H_0$. Whenever we do not make any assumptions on the joint distribution of $\mathbf{x}$ and $y$, we denote it by $H$.

The estimators in Section \ref{sec:Intro} are all defined at the sample level. To derive their influence function, we first need to introduce their equivalents at the population level. For the penalized M-estimator (\ref{eq:betaM_data}), the corresponding definition at the population level, with $(\mathbf{x}, y)\sim H$, is
\begin{align}
\mathbf{\beta}_M(H)=\argmin_{\mathbf{\beta}\in\mathbb{R}^p} \mathbb{E}_H\big[\rho(y-\mathbf{x}'\mathbf{\beta}) \big]+ 2\lambda\sum_{j=1}^p J(\beta_j)
\label{eq:betaM}
\end{align}
An example of a penalized M-estimator is the ridge functional, for which $\rho(z) = J(z) = z^2$. Also the lasso functional 
\begin{align}
\label{eq:lassoMulti}
\mathbf{\beta}_{LASSO}(H)=\argmin_{\mathbf{\beta}\in\mathbb{R}^p} \left(\mathbb{E}_H[(y-\mathbf{x}'\mathbf{\beta})^2]+2\lambda\sum_{i=1}^p|\beta_i|\right)
\end{align}
can be seen as a special case of the penalized M-estimator. However, its penalty is not differentiable, which will cause problems in the computation of the influence function. 

To create more robust functionals, different loss functions than the classical quadratic loss function $\rho(z) = z^2$ can be considered. Popular choices are the Huber function
\begin{align}
\label{eq:Huber}
\rho_H(z) = \begin{cases}
z^2 &\text{ if } |z|\leq k_H,\\
2k_H|z|-k_H^2 &\text{ if } |z|>k_H
\end{cases}
\end{align}
and Tukey's biweight function
\begin{align}
\label{eq:biweight}
\rho_{BI}(z) = \begin{cases}
1-(1-(\frac{z}{k_{BI}})^2)^3 &\text{ if } |z|\leq k_{BI},\\
1 &\text{ if } |z|>k_{BI}.
\end{cases}
\end{align}
The Huber loss function $\rho_H$ is a continuous, differentiable function that is quadratic in a central region $[-k_H,k_H]$ and increases only linearly outside of this interval (compare Figure~\ref{fig:biweight}). The function value of extreme residuals is therefore lower than with a quadratic loss function and, as a consequence, those observations have less influence on the estimate. Due to the quadratic part in the central region, the Huber loss function is still differentiable at zero in contrast to an absolute value loss. The main advantage of the biweight function $\rho_{BI}$ (sometimes also called `bisquared' function) is that it is a smooth function that trims large residuals, while small residuals receive a function value that is similar as with a quadratic loss (compare Figure~\ref{fig:biweight}). The choice of the tuning constants $k_{BI}$ and $k_H$ determines the breakdown point and efficiency of the functionals. We use $k_{BI}=4.685$ and $k_H=1.345$, which gives 95\% of efficiency for a standard normal error distribution in the unpenalized case. To justify the choice of $k$ also for distributions with a scale different from $1$, the tuning parameter has to be adjusted to $k\hat{\sigma}$.

\begin{figure}
\begin{center}
	\includegraphics[width = \textwidth]{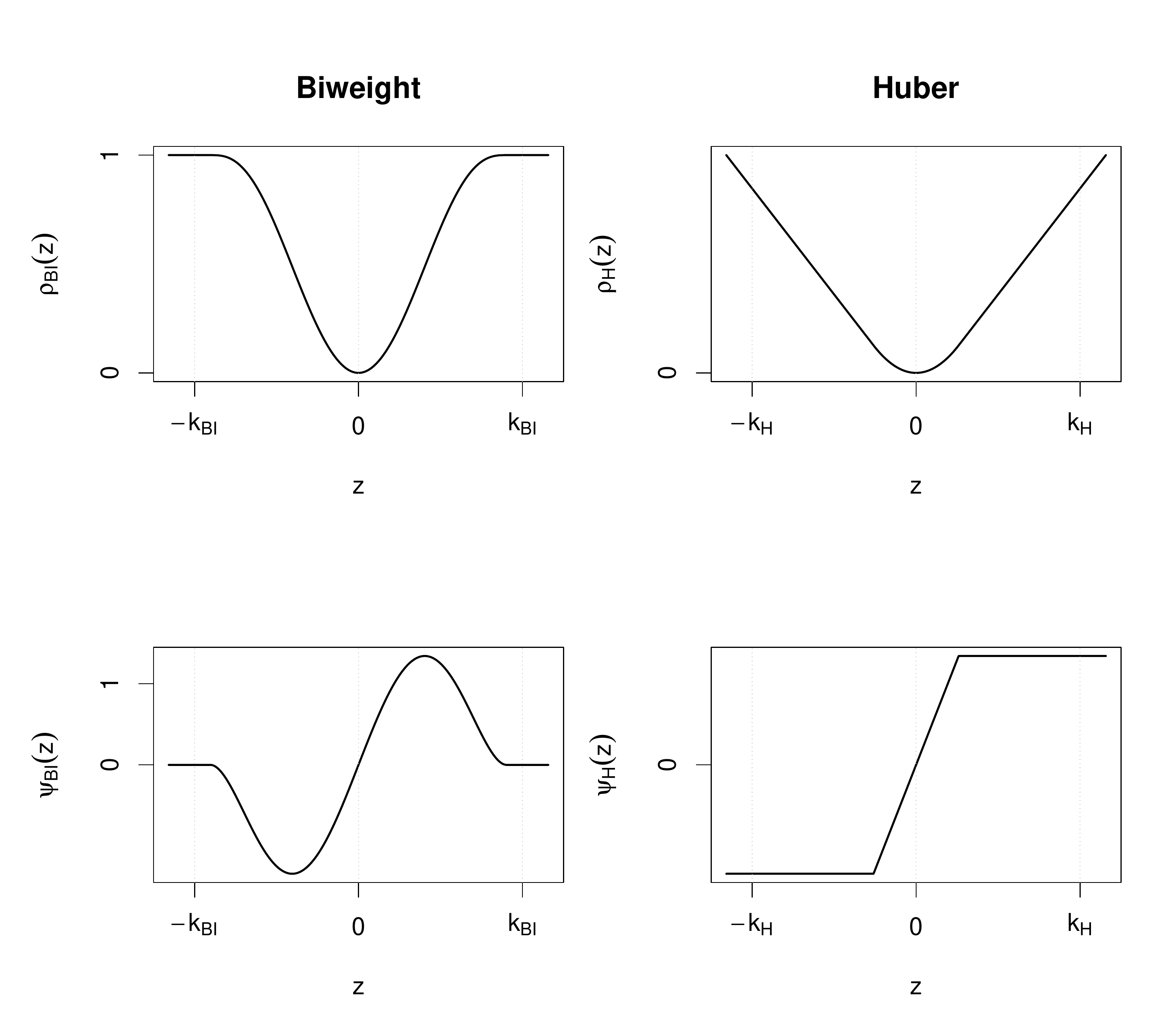}
\caption{Biweight and Huber loss function $\rho$ and their first derivatives $\psi$.}
\label{fig:biweight}
\end{center}
\end{figure}	

Apart from the $L_1$- and $L_2$-penalty used in lasso an ridge estimation, respectively, also other penalty functions can be considered. Another popular choice is the smoothly clipped absolute deviation (SCAD) penalty \citep[][]{Fan} (see Figure~\ref{fig:SCAD_2})
\begin{align}
\label{eq:SCAD}
J_{SCAD}(\beta) = \begin{cases}
|\beta| & \text{ if } |\beta|\leq\lambda,\\
-\frac{(|\beta|-a\lambda)^2}{2(a-1)\lambda}+\lambda\frac{a+1}{2} & \text{ if } \lambda<|\beta|\leq a\lambda,\\
\lambda\frac{a+1}{2} & \text{ if } |\beta|>a\lambda.
\end{cases}
\end{align}
While the SCAD functional, exactly as the lasso, shrinks (with respect to $\lambda$) small parameters to zero, large values are not shrunk at all, exactly as in least squares regression.

\begin{figure}
\begin{center}
	\includegraphics[width = 0.7\textwidth]{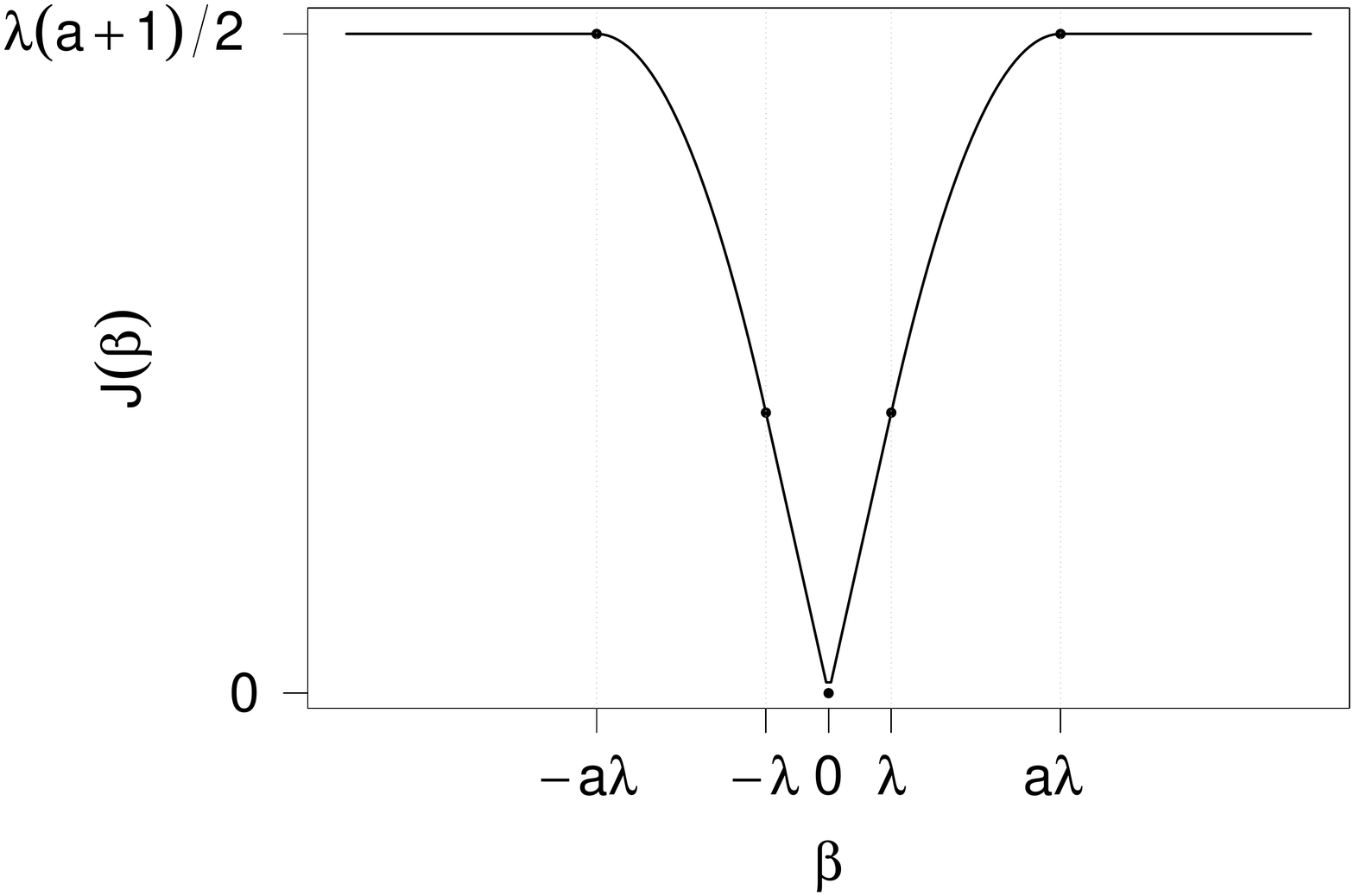}
\caption{The smoothly clipped absolute deviation (SCAD) penalty function}
\label{fig:SCAD_2}
\end{center}
\end{figure}

The definition of the sparse LTS estimator at a population level is
\begin{align}
\mathbf{\beta}_{spLTS}(H) = \argmin_{\mathbf{\beta}\in\mathbb{R}^p} \mathbb{E}_H[(y-\mathbf{x}'\mathbf{\beta})^2 I_{[|y-\mathbf{x}'\mathbf{\beta}|\leq q_{\mathbf{\beta}}]}] + \alpha \lambda \sum_{j=1}^p|\beta_j|,
\label{eq:spLTS}
\end{align}
with $q_{\mathbf{\beta}}$ the $\alpha$-quantile of $|y-\mathbf{x}'\mathbf{\beta}|$. As recommended in \cite{Alfons}, we take $\alpha = 0.75$.


\section{Bias}
\label{sec:Bias}

The penalized M-functional $\mathbf{\beta}_M$ has a bias 
\begin{align}
\Bias(\mathbf{\beta}_M,H_0)&=\mathbf{\beta}_M(H_0)-\mathbf{\beta}_0
\label{eq:bias}
\end{align}
at the model distribution $H_0$. The bias is due to the penalization and is also present for penalized least squares functionals. Note that there is no bias for non-penalized M-functionals. The difficulty of Equation (\ref{eq:bias}) lies in the computation of the functional $\mathbf{\beta}_M(H_0)$. For the lasso functional, there exists an explicit solution only for simple regression (i.e. $p=1$)
\begin{align}
\label{eq:lassoFunct}
\beta_{LASSO}(H) = \sign(\beta_{LS}(H))\bigg(\left|\beta_{LS}(H)\right|-\frac{\lambda}{\mathbb{E}_H[x^2]}\bigg)_+.
\end{align}
Here $\beta_{LS}(H)=\mathbb{E}_H[xy]/\mathbb{E}_H[x^2]$ denotes the least squares functional and $(z)_+=\max(0,z)$, the positive part function. For completeness, we give a proof of Equation (\ref{eq:lassoFunct}) in the appendix. For multiple regression the lasso functional at the model distribution $H_0$ can be computed using the idea of the coordinate descent algorithm (see Section \ref{sec:IF_lasso}), with the model parameter $\mathbf{\beta}_0$ as a starting value. Similarly, also for the SCAD functional there exists an explicit solution only for simple regression
\begin{align}
\beta_{SCAD}(H) = \begin{cases}
(|\beta_{LS}(H)|-\frac{\lambda}{\mathbb{E}_{H_0}[x^2]})_+ \sign(\beta_{LS}(H)) & \text{ if } |\beta_{LS}(H)|\leq \lambda+\frac{\lambda}{\mathbb{E}_{H_0}[x^2]},\\
\frac{(a-1)\mathbb{E}_{H_0}[x^2]\beta_{LS}(H) - a\lambda\sign(\beta_{LS}(H))}{(a-1)\mathbb{E}_{H_0}[x^2]-1} & \text{ if } \lambda + \frac{\lambda}{\mathbb{E}_{H_0}[x^2]}<|\beta_{LS}(H)| \leq a\lambda,\\
\beta_{LS}(H) & \text{ if } |\beta_{LS}(H)| > a\lambda.
\end{cases}
\label{eq:betascad}
\end{align}
This can be proved using the same ideas as in the computation of the solution for the lasso functional in simple regression (see Proof of Equation (\ref{eq:lassoFunct}) in the appendix). Here the additional assumption $\mathbb{E}_H[x^2]>1/(a-1)$ is needed. As can be seen from Equation (\ref{eq:betascad}), the SCAD functional is unbiased at the model $H_0$ for large values of the parameter $\mathbf{\beta}_0$.

To compute the value of a penalized M-functional that does not use a quadratic loss function, the iteratively reweighted least squares (IRLS) algorithm \citep[][]{Osborne} can be used to find a solution. Equation (\ref{eq:betaM}) can be rewritten as 
\begin{align*}
\mathbf{\beta}_M(H)=\argmin_{\mathbf{\beta}\in\mathbb{R}^p}\mathbb{E}_H[w(\mathbf{\beta})(y-\mathbf{x}'\mathbf{\beta})^2] + 2\lambda\sum_{j=1}^p J(\beta_j)
\end{align*}
with weights $w(\mathbf{\beta})=\rho(y-\mathbf{x}'\mathbf{\beta})/(y-\mathbf{x}'\mathbf{\beta})^2$. If a value of $\mathbf{\beta}$ is available, the weights can be computed. If the weights are taken as fixed, $\mathbf{\beta}_M$ can be computed using a weighted lasso (if an $L_1$-penalty was used), weighted SCAD (for a SCAD-penalty) or a weighted ridge (if an $L_2$-penalty is used). Weighted lasso and weighted SCAD can be computed using a coordinate descent algorithm, for the weighted ridge an explicit solution exists. Computing weights and $\mathbf{\beta}_M$ iteratively, convergence to a local solution of the objective function will be reached. As a good starting value we take the true value $\mathbf{\beta}_0$. The expected values that are needed for the weighted lasso/SCAD/ridge are calculated by Monte Carlo approximation.

For the sparse LTS functional, we can find an explicit solution for simple regression with normal predictor and error term.
\begin{lemma}
\label{lemma:spLTS_uni}
Let $y=x\beta_0 + e$ be a simple regression model as in (\ref{eq:ModelF}). Let $H_0$ be the joint distribution of $x$ and $y$, with $x$ and $e$ normally distributed. Then the explicit solution of the sparse LTS functional (\ref{eq:spLTS}) is 
\begin{align}
\beta_{spLTS}(H_0) = \sign(\beta_0)\left(|\beta_0| - \frac{\alpha\lambda}{2 c_1 \mathbb{E}_{H_0}[x^2]}\right)_+
\label{eq:spLTS_uni}
\end{align}
with $c_1 = \alpha - 2q_\alpha \phi(q_\alpha)$, $q_\alpha$ the $\frac{\alpha+1}{2}$-quantile of the standard normal distribution and $\phi$ its density.
\end{lemma}

Lemma \ref{lemma:spLTS_uni} gives an explicit solution of the sparse LTS functional for only normally distributed errors and predictors, which is a strong limitation. In the general case, with $x\sim F$, $e\sim G$, and $x$ and $e$ independent, the residual $y-x\beta=x(\beta_0-\beta)+e$ follows a distribution $D_\beta(z)=F(z/(\beta_0-\beta))\ast G(z)$ for $\beta_0>\beta$, where $\ast$ denotes the convolution. Without an explicit expression for $D_\beta$, it will be hard to obtain an explicit solution for the sparse LTS functional. On the other hand, if $D_\beta$ is explicitly known, the proof of Lemma $\ref{lemma:spLTS_uni}$ can be followed and an explicit solution for the sparse LTS-functional can be found.
A case where explicit results are feasible is for $x$ and $e$ both Cauchy distributed, since the convolution of Cauchy distributed variables remains Cauchy. Results for this case are available from the first author upon request.

To study the bias of the various functionals of Section \ref{sec:Funct}, we take $p=1$ and assume $x$ and $e$ as standard normally distributed. We use $\lambda=0.1$ for all functionals. Figure~\ref{fig:Bias} displays the bias as a function of $\beta_0$. Of all functionals used only least squares has a zero bias. The $L_1$-penalized functionals have a constant bias for values of $\beta_0$ that are not shrunken to zero. For smaller values of $\beta_0$ the bias increases monotonously in absolute value. Please note that the penalty parameter $\lambda$ plays a different role for different estimators, as the same $\lambda$ yields different amounts of shrinkage for different estimators. For this reason, Figure~\ref{fig:Bias} illustrates only the general shape of the bias as a function of $\beta_0$.

\begin{figure}
\begin{center}
	\includegraphics[width = 0.7\textwidth]{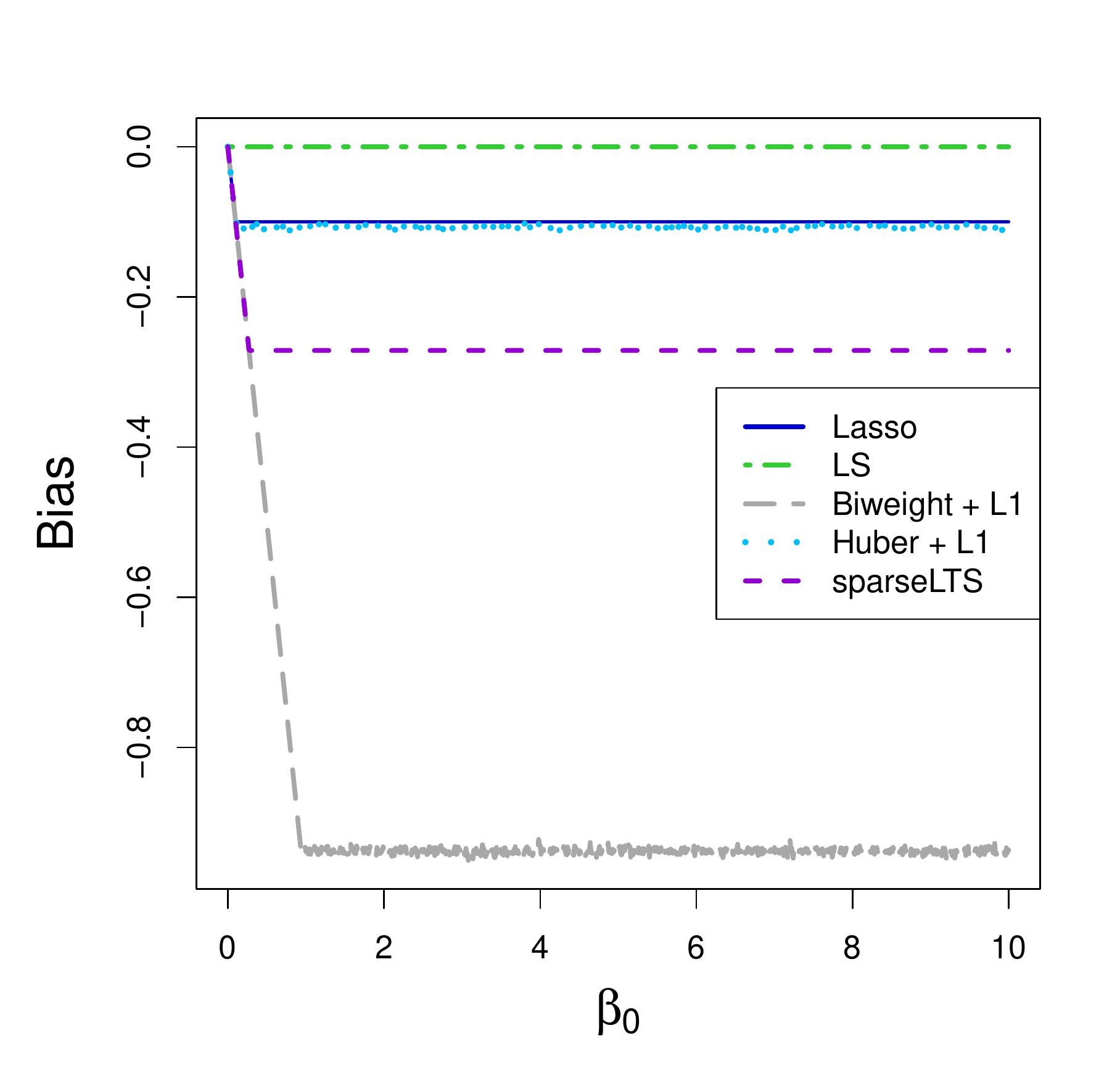}
\caption{Bias of various functionals for different values of $\beta_0$ ($\lambda = 0.1$ fixed). Note that the small fluctuations are due to Monte Carlo simulations in the computation of the functional.}
\label{fig:Bias}
\end{center}
\end{figure}


\section{The Influence Function}
\label{sec:IF}

The robustness of a functional $\mathbf{\beta}$ can be measured via the influence function
\begin{align*}
IF((\mathbf{x}_0,y_0),\mathbf{\beta},H) = \frac{\partial}{\partial\epsilon} \bigg[\mathbf{\beta}\big((1-\epsilon) H + \epsilon \delta_{(\mathbf{x}_0,y_0)}\big)\bigg]\bigg|_{\epsilon=0}.
\end{align*}
It describes the effect of infinitesimal, pointwise contamination in $(\mathbf{x}_0,y_0)$ on the functional $\mathbf{\beta}$. Here $H$ denotes any distribution and $\delta_{\mathbf{z}}$ the point mass distribution at $\mathbf{z}$. To compute the influence function of the penalized M-functional (\ref{eq:betaM}), smoothness conditions for functions $\rho(\cdot)$ and $J(\cdot)$ have to be assumed.

\begin{prop}
\label{prop:gen}
Let $y=\mathbf{x}'\beta_0+e$ be a regression model as defined in (\ref{eq:ModelF}). Furthermore, let $\rho,J:\mathbb{R}\rightarrow\mathbb{R}$ be twice differentiable functions and denote the derivative of $\rho$ by  $\psi := \rho'$. Then the influence function of the penalized M-functional $\mathbf{\beta}_M$ for $\lambda\geq0$ is given by
\begin{align}
IF((\mathbf{x}_0,&y_0),\mathbf{\beta}_M, H_0)=\notag\\
=&\big(\mathbb{E}_{H_0}[\psi'(y-\mathbf{x'\,\beta}_{M}(H_0))\mathbf{x}\mathbf{x}']+2\lambda \diag(J''(\mathbf{\beta}_M(H_0)))\big)^{-1}\cdot \label{eq:IFgen} \\
& \cdot \big(\psi(y_0-\mathbf{x}'_0\mathbf{\beta}_M(H_0))\mathbf{x}_0 - \mathbb{E}_{H_0}[\psi(y-\mathbf{x'\,\beta}_{M}(H_0))\mathbf{x}]\big).\notag
\end{align}
\end{prop}
\noindent The influence function (\ref{eq:IFgen}) of the penalized M-functional is unbounded in $\mathbf{x}_0$ and is only bounded in $y_0$ if $\psi(\cdot)$ is bounded. In Section \ref{sec:Plot_IF} we will see that the effect of the penalty on the shape of the influence function is quite small compared to the effect of the loss function. 

As the ridge functional can be seen as a special case of the penalized M-functional (\ref{eq:betaM}), its influence function follows as a corollary:
\begin{cor}
The influence function of the ridge functional $\mathbf{\beta}_{RIDGE}$ is
\begin{align}
IF(&(\mathbf{x}_0, y_0),\mathbf{\beta}_{RIDGE}, H_0) = \notag\\
&\big(\mathbb{E}_{H_0}[\mathbf{xx}']+2\lambda I_p\big)^{-1}\bigg(\big(y_0-\mathbf{x}_0'\mathbf{\beta}_{RIDGE}(H_0)\big)\mathbf{x}_0 + \mathbb{E}_{H_0}\big[\mathbf{x}\mathbf{x}'\big]\Bias(\mathbf{\beta}_{RIDGE},H_0)\bigg).
\label{eq:IFRidge}
\end{align}
\end{cor}
\noindent As the function $\psi(z)=2z$ is unbounded, the influence function (\ref{eq:IFRidge}) of the ridge functional is unbounded. Thus the ridge functional is not robust to any kind of outliers.

The penalty function $J(z):=|z|$ of the lasso functional and the sparse LTS functional is not twice differentiable at zero. Therefore those functionals are no special cases of the M-functional used in Proposition \ref{prop:gen} and have to be considered separately to derive the influence function.


\section{The Influence Function of the Lasso}
\label{sec:IF_lasso}

For simple regression, i.e. for $p=1$, an explicit solution for the lasso functional exists, see Equation (\ref{eq:lassoFunct}). With that the influence function can be computed easily.

\begin{lemma}
\label{lemma:lasso_uni}
Let $y=x\beta_0+e$ be a simple regression model as in (\ref{eq:ModelF}). Then the influence function of the lasso functional is
\begin{align}
IF((x_0,y_0), \beta_{LASSO}, H_0) = \begin{cases}
0 &\text{ if }-\frac{\lambda}{\mathbb{E}_{H_0}[x^2]}\leq \beta_0<\frac{\lambda}{\mathbb{E}_{H_0}[x^2]}\\
\frac{x_0(y_0-\beta_0 x_0)}{\mathbb{E}_{H_0}[x^2]}-\lambda\frac{\mathbb{E}_{H_0}[x^2]-x_0^2}{\left(\mathbb{E}_{H_0}[x^2]\right)^2}\sign(\beta_0) &\text{ otherwise}.
\end{cases}
\label{eq:IF_lasso_uni}
\end{align}
\end{lemma}
\noindent Similar to the influence function of the ridge functional (\ref{eq:IFRidge}), the influence function of the lasso functional (\ref{eq:IF_lasso_uni}) is unbounded in both variables $x_0$ and $y_0$ in case the coefficient $\beta_{LASSO}$ is not shrunk to zero (Case $2$ in Equation (\ref{eq:IF_lasso_uni})). Otherwise the influence function is constantly zero. The reason of the similarity of the influence function of the lasso and the ridge functional is that both are a shrunken version of the least squares functional. 

As there is no explicit solution in multiple regression for the lasso functional, its influence function cannot be computed easily. However, \cite{Friedman} and \cite{Fu} found an algorithm, the \textit{coordinate descent algorithm} (also \textit{shooting algorithm}), to split up the multiple regression into a number of simple regressions. The idea of the coordinate descent algorithm at population level is to compute the lasso functional (\ref{eq:lassoMulti}) variable by variable. Repeatedly, one variable $j\in\{1,\ldots,p\}$ is selected. The value of the functional $\beta_j^{cd}$ is then computed holding all other coefficients $k\neq j$ fixed at their previous value $\beta_k^\ast$
\begin{align}
\beta_j^{cd}(H) &= \argmin_{\beta_j\in \mathbb{R}}\mathbb{E}_H[((y-\sum_{k\neq j} x_k\beta_k^\ast) - x_j\beta_j)^2] + 2\lambda\sum_{k\neq j} |\beta_k^\ast| + 2\lambda|\beta_j| \notag\\
&=\argmin_{\beta_j\in \mathbb{R}}\mathbb{E}_H[((y-\sum_{k\neq j} x_k\beta_k^\ast) - x_j\beta_j)^2] + 2\lambda|\beta_j|.\label{eq:lassoshoot}
\end{align}
This can be seen as simple lasso regression with partial residuals $y-\sum_{k\neq j} x_k\beta^\ast_k$ as response and the $j$th coordinate $x_j$ as covariate. Thus, the new value of $\beta_j^{cd}(H)$ can be easily computed using Equation (\ref{eq:lassoFunct}). Looping through all variables repeatedly, convergence to the lasso functional (\ref{eq:lassoMulti}) will be reached for any starting value \citep[][]{Friedman, Tseng}.

For the coordinate descent algorithm an influence function can be computed similarly as for simple regression. However, now the influence function depends on the influence function of the previous value $\mathbf{\beta}^\ast$.

\begin{lemma}
\label{lemma:lasso_shooting}
Let $y=\mathbf{x}'\beta_0+e$ be the regression model of (\ref{eq:ModelF}). Then the influence function of the jth coordinate of the lasso functional (\ref{eq:lassoshoot}) computed via coordinate descent is
\begin{align}
IF(&(\mathbf{x}_0,y_0),\beta^{cd}_j,H_0)=\begin{cases}
&\hspace{-0.3cm}0 \hspace{4.5cm}\text{ if }\quad\left|\mathbb{E}_{H_0}[x_j \tilde{y}^{(j)}]\right| < \lambda,\\
&\hspace{-0.3cm}\frac{-\mathbb{E}_{H_0}[x_j{\mathbf{x}^{(j)}}'IF((\mathbf{x}_0,y_0), \,\mathbf{\beta^\ast}^{(j)}, H_0)] + (y_0-{\mathbf{x}_0^{(j)}}'\mathbf{\beta^\ast}^{(j)}(H_0)) (\mathbf{x}_0)_j}{\mathbb{E}_{H_0}[x_j^2]} - \frac{\mathbb{E}_{H_0}[x_j\tilde{y}^{(j)}](\mathbf{x}_0)_j^2}{(\mathbb{E}_{H_0}[x_j^2])^2}\\
&\hspace{2cm}-\lambda\frac{\mathbb{E}_{H_0}[x_j^2]-(\mathbf{x}_0)_j^2}{\left(\mathbb{E}_{H_0}[x_j^2]\right)^2}\,\sign(\mathbb{E}_{H_0}[x_j\tilde{y}^{(j)}]) \text{ otherwise},
\end{cases}
\label{eq:IF_shoot}
\end{align}
where for any vector $\mathbf{z}$ we define $\mathbf{z}^{(j)}= (z_1,\ldots, z_{j-1},z_{j+1},\ldots,z_p)'$, $\tilde{y}^{(j)}:= y - {\mathbf{x}^{(j)}}'{\mathbf{\beta}^\ast}^{(j)}(H_0)$, with ${\mathbf{\beta}^\ast}^{(j)}$ the functional representing the value of the coordinate descent algorithm at population level in the previous step.
 \end{lemma}

\noindent To obtain a formula for the influence function of the lasso functional in multiple regression, we can use the result of Lemma \ref{lemma:lasso_shooting}. The following proposition holds.

\begin{prop}
\label{prop:lasso_multi}
Let $y=\mathbf{x}'\beta_0+e$ be the regression model of (\ref{eq:ModelF}). Without loss of generality let $\mathbf{\beta}_{LASSO}(H_0)=((\mathbf{\beta}_{LASSO}(H_0))_1,\ldots,(\mathbf{\beta}_{LASSO}(H_0))_k,0,\ldots,0)'$ with $k\leq p$ and $(\mathbf{\beta}_{LASSO}(H_0))_j\neq 0\,\forall j=1,\ldots,k$. Then the influence function of the lasso functional (\ref{eq:lassoMulti}) is 
\begin{align}
IF(&(\mathbf{x}_0,y_0), \mathbf{\beta}_{LASSO},H_0) =\label{eq:IF_lasso_multi}\\
&=\begin{pmatrix} \left(\mathbb{E}_{H_0}[\mathbf{x}_{1:k}\mathbf{x}_{1:k}']\right)^{-1}\bigg((\mathbf{x}_0)_{1:k}(y_0-\mathbf{x}_0'\mathbf{\beta}_{LASSO}(H_0))-\mathbb{E}_{H_0}[\mathbf{x}_{1:k}(y-\mathbf{x}'\mathbf{\beta}_{LASSO}(H_0))]\bigg)\\
\mathbf{0}_{p-k}
\end{pmatrix}\notag
\end{align}
with the notation $\mathbf{z}_{r:s}=(z_r, z_{r+1},\ldots,z_{s-1}, z_s)'$ for $\mathbf{z}\in\mathbb{R}^p$, $r,s\in\{1,\ldots,p\}$ and $r\leq s$.
\end{prop}

Thus, the influence function of the lasso estimator is zero for variables $j$ with coefficients $(\mathbf{\beta}_{LASSO}(H_0))_j$ shrunk to zero. This implies that for an infinitesimal amount of contamination, the lasso estimator in those variables $j$ stays $(\mathbf{\beta}_{LASSO}(H_0))_j=0$ and is not affected by the contamination.

Another approach to compute the influence function of the lasso functional is to consider it as a limit case of functionals satisfying the conditions of Proposition \ref{prop:gen}. The following sequence of hyperbolic tangent functions converges to the sign-function
\begin{align*}
\lim_{K\rightarrow\infty} \tanh(K x) =
\begin{cases}
+1 & \text{ if } x>0,\\
-1 & \text{ if } x<0,\\
0 & \text{ if } x=0.
\end{cases}
\end{align*}
Hence, it can be used to get a smooth approximation of the absolute value function
\begin{align}
|x| = x\cdot\sign(x) = \lim_{K\rightarrow\infty} x\cdot \tanh(Kx).
\label{eq:tanh_approx}
\end{align}
The larger the value of $K>1$, the better the approximation becomes (see Figure~\ref{fig:tanh}).
\begin{figure}
\begin{center}
	\includegraphics[width = 0.65\textwidth]{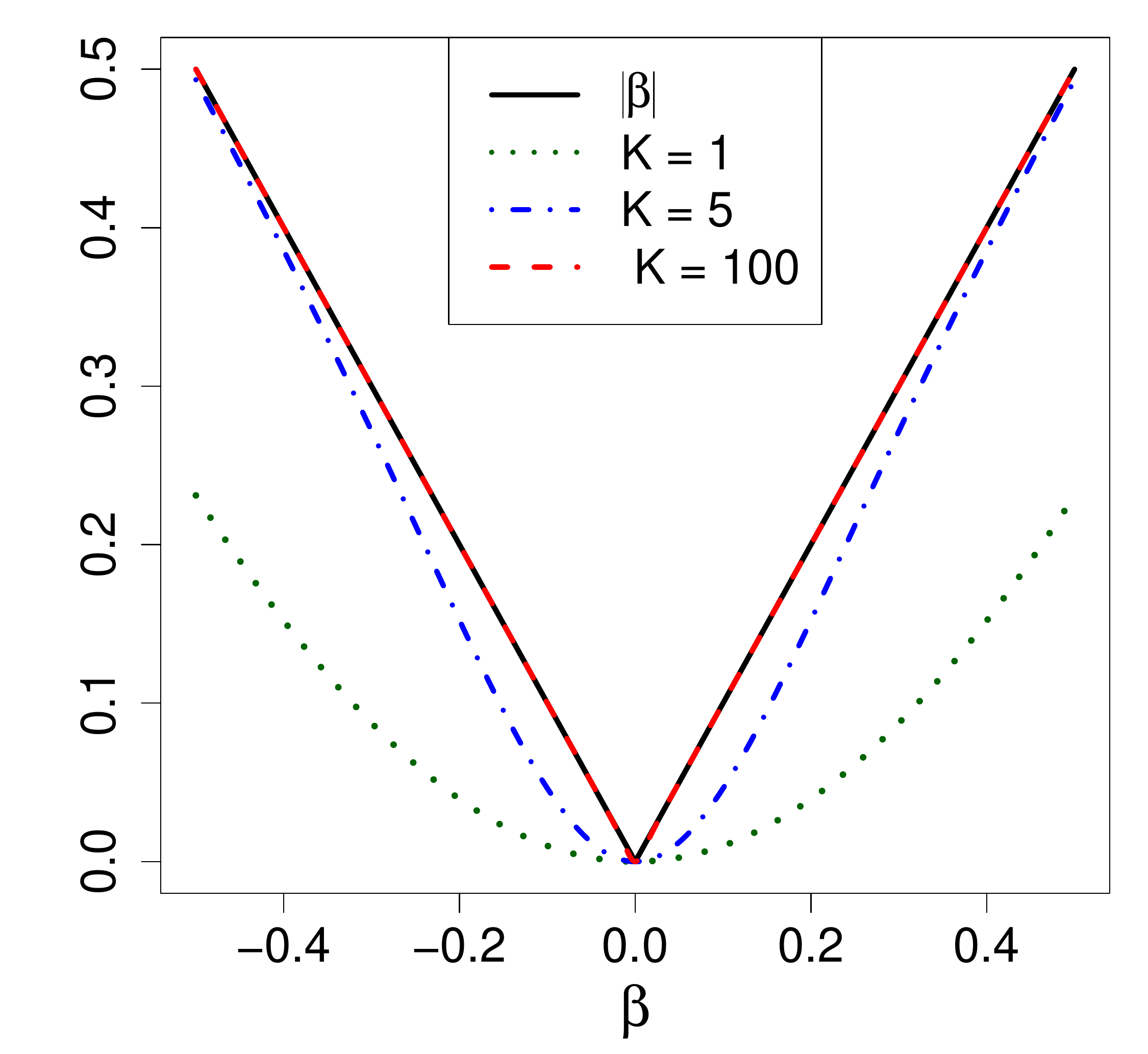}
\caption{Approximation of $|\beta|$ using $\beta\cdot\tanh(K\beta)$}
\label{fig:tanh}
\end{center}
\end{figure}	
Therefore the penalty function $J_K(\beta_j) = \beta_j\tanh(K\beta_j)$ is an approximation of $J_{LASSO}(\beta_j) = |\beta_j|$. As $J_K$ is a smooth function, the influence function of the corresponding functional 
\begin{align}
\mathbf{\beta}_K(H_0) = \argmin_{\mathbf{\beta}\in\mathbb{R}^p} \mathbb{E}_{H_0}[(y-\mathbf{x}'\mathbf{\beta})^2] + 2\lambda \sum_{j=1}^p J_K(\beta_j)
\label{eq:beta_tan}
\end{align}
can be computed by applying Proposition \ref{prop:gen}. Taking the limit of this influence function, we obtain the influence function of the lasso functional. It coincides with the expression given in Proposition \ref{prop:lasso_multi}.

\begin{lemma}
\label{lemma:approx_tan}
Let $y=\mathbf{x}'\beta_0+e$ be the regression model of (\ref{eq:ModelF}). Without loss of generality let $\mathbf{\beta}_{LASSO}(H_0)=((\mathbf{\beta}_{LASSO}(H_0))_1,\ldots,(\mathbf{\beta}_{LASSO}(H_0))_k,0,\ldots,0)'$ with $k\leq p$ and $(\mathbf{\beta}_{LASSO}(H_0))_j\neq 0\,\forall j=1,\ldots,k$. Then the influence function of the penalized M-estimator (\ref{eq:beta_tan}) converges to the influence function of the lasso functional given in (\ref{eq:IF_lasso_multi}) as $K$ tends to infinity.
\end{lemma}


\section{The Influence Function of sparse LTS}
\label{sec:IF_spLTS}

For sparse LTS, computation of the influence function is more difficult than for the lasso. In addition to the nondifferentiable penalty function, sparse LTS also has a discontinuous loss function. For simplicity, we therefore assume a univariate normal distribution for the predictor $x$ and the error $e$. However, the below presented ideas can be used to derive the influence function also for other distributions (similar as stated below Lemma \ref{lemma:spLTS_uni}). Results for Cauchy distributed predictors and errors are available from the first author upon request.

\begin{lemma}
\label{lemma:IFspLTS}
Let $y=x\beta_0+e$ be a simple regression model as in (\ref{eq:ModelF}). If $x$ and $e$ are normally distributed, the influence function of the sparse LTS functional (\ref{eq:spLTS_uni}) is 
\begin{align}
IF(&(x_0, y_0), \beta_{spLTS},H_0) =\begin{cases}&\hspace{-0.3cm}0 \hspace{4.5cm} \text{ if } -\frac{\alpha \lambda}{2c_1\mathbb{E}_{H_0}[x^2]} < \beta_0 \leq \frac{\alpha \lambda}{2c_1\mathbb{E}_{H_0}[x^2]},\\
& \hspace{-0.3cm}(\beta_{spLTS}(H_0) - \beta_0) - \frac{q_\alpha^2(I_{[|r_0|\leq q_\alpha] }-\alpha)(\beta_0-\beta_{spLTS}(H_0))}{\alpha - 2 q_\alpha\phi(q_\alpha)} + \\
& \hspace{-0.3cm}\hspace{3.5cm}+\frac{x_0(y_0-x_0\beta_{spLTS}(H_0))I_{[|r_0|\leq q_\alpha] }}{(\alpha - 2 q_\alpha\phi(q_\alpha))\mathbb{E}_{H_0}[x^2]}  \text{ otherwise}
\end{cases}
\label{eq:IFspLTS}
\end{align}
with $r_0 = \frac{y_0 - x_0\beta_{spLTS}(H_0)}{\sqrt{\sigma^2+(\beta_0-\beta_{spLTS}(H_0))^2\mathbb{E}_{H_0}[x^2]}}$ and the same notation as in Lemma \ref{lemma:spLTS_uni}. 
\end{lemma}

Lemma \ref{lemma:IFspLTS} shows that the influence function of the sparse LTS functional may become unbounded for points $(x_0, y_0)$ that follow the model, i.e. for good leverage points, but remains bounded elsewhere, in particular for bad leverage points and vertical outliers. This shows the good robust properties of sparse LTS.

We can also see from Equation (\ref{eq:IFspLTS}) that the influence function of the sparse LTS functional is zero if the functional is shrunken to zero, i.e. if $|\beta_0| \leq \frac{\alpha \lambda}{2c_1\mathbb{E}_{H_0}[x^2]}$. This result is the same as for the lasso functional (see Proposition \ref{prop:lasso_multi}). It implies that infinitesimal amounts of contamination do not affect the functional, when the latter is shrunken to zero.


\section{Plots of Influence Functions}
\label{sec:Plot_IF}

We first compare the effects of different penalties and take a quadratic loss function. We consider least squares, ridge and lasso regression as well as the SCAD penalty (\ref{eq:SCAD}). To compute ridge and lasso regression a value for the penalty parameter $\lambda$ is needed, and for SCAD another additional parameter $a$ has to be specified. We choose a fixed value $\lambda = 0.1$ and, as proposed by \cite{Fan}, we use $a=3.7$. 

Influence functions can only be plotted for simple regression $y=x\beta_0+e$, i.e. for $p=1$. We specify the predictor and the error as independent and standard normally distributed. For the parameter $\beta_0$ we use a parameter $\beta_0 = 1.5$ that will not be shrunk to zero by any of the functionals, as well as $\beta_0=0$ to focus also on the sparseness of the functionals. Figures~\ref{fig:IF_penalty_beta15} and \ref{fig:IF_penalty_beta0} show the plots of the influence functions for least squares, ridge, lasso and SCAD for both values of $\beta_0$. Examining Figure~\ref{fig:IF_penalty_beta15}, one could believe that all influence functions are equal. The same applies for the influence functions of least squares and ridge in Figure~\ref{fig:IF_penalty_beta0}. However, this is not the case. All influence functions are different of one another because their bias and the second derivative of the penalty appear in the expression of the influence function. Those terms are different for the different functionals. Usually, the differences are minor. Note, however, that for some specific choices of $\lambda$ and $\beta_0$ differences can be substantial. For $\beta_0 = 0$, see Figure~\ref{fig:IF_penalty_beta0}, SCAD and lasso produce a constantly zero influence function. We may conclude that in most cases the effect of the penalty function on the shape of the influence function is minor.

\begin{figure}
\vspace{-0.8cm}
\begin{center}
	\includegraphics[width = 0.88\textwidth]{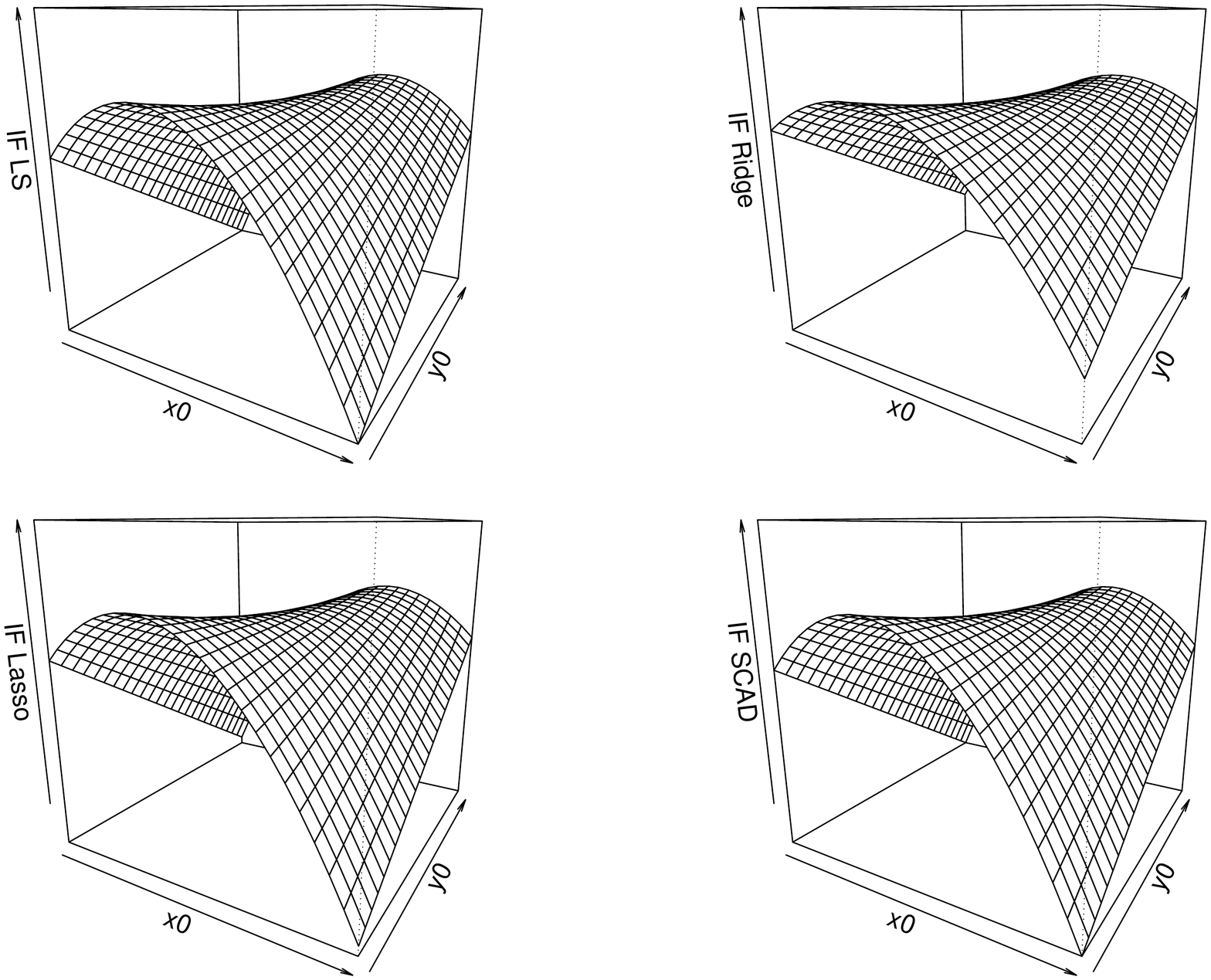}
\caption{Influence functions for different penalty functions (least squares, ridge, lasso and SCAD) for $\beta_0 = 1.5$ with $(x_0, y_0)\in[-10,10]^2$ and the vertical axis ranging from $-250$ to $100$}
\label{fig:IF_penalty_beta15}
\centering
	\includegraphics[width = 0.88\textwidth]{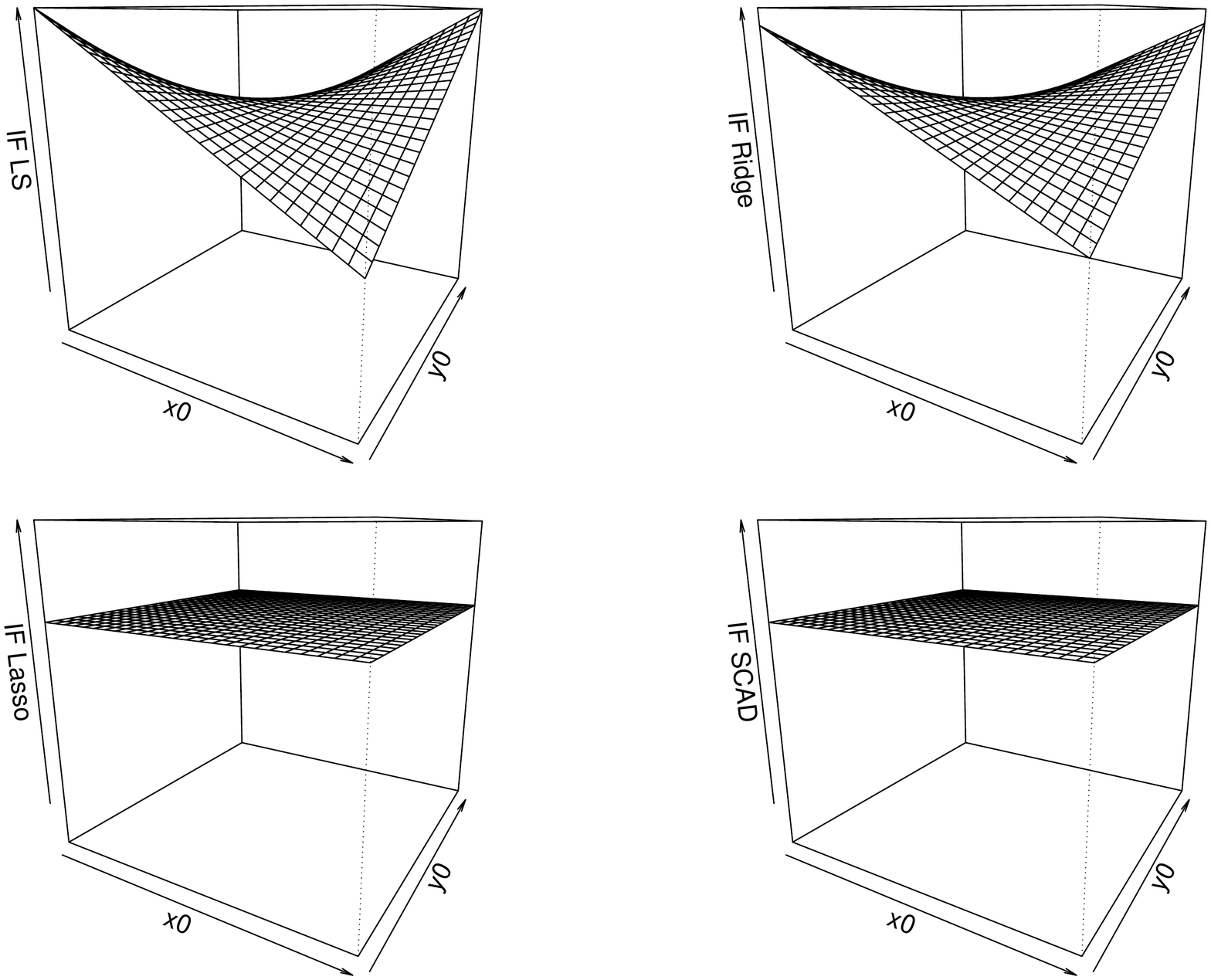}
\caption{Influence functions for different penalty functions (least squares, ridge, lasso and SCAD) for $\beta_0 = 0$ with $(x_0, y_0)\in[-10,10]^2$ and the vertical axis ranging from $-250$ to $100$}
\label{fig:IF_penalty_beta0}
\end{center}
\end{figure}

To compare different loss functions, we use Huber loss (\ref{eq:Huber}), biweight loss (\ref{eq:biweight}) and sparse LTS (\ref{eq:spLTS}), each time combined with the $L_1$-penalty $J(\beta)=|\beta|$ to achieve sparseness. For the simple regression model $y=x\beta_0+e$, we specify the predictor and the error as independent and standard normally distributed and consider $\beta_0=0$ and $\beta_0=1.5$. Furthermore, we fix $\lambda = 0.04$. 

Figure~\ref{fig:IF_loss} shows the influence functions of these functionals with Huber and biweight loss function. They clearly differ from the ones using the classic quadratic loss for coefficients $\beta_0$ that are not shrunk to zero (compare to panels corresponding to the lasso in Figures~\ref{fig:IF_penalty_beta0} and \ref{fig:IF_penalty_beta15}). The major difference is that the influence functions of functionals with a bounded loss function (sparse LTS, biweight) are only unbounded for good leverage points and bounded for regression outliers. This indicates the robust behavior of the functionals. It is even further emphasized by the fact that those observations $(x_0,y_0)$ with big influence are the ones with small residuals $y_0 - x_0\beta_0$, that is the ones that closely follow the underlying model distribution. Observations with large residuals have small and constant influence. In contrast, the unbounded Huber loss function does not achieve robustness against all types of outliers. Only for outliers in the response the influence is constant (for a fixed value of $x_0$). However, if the predictor values increase, the influence of the corresponding observation increases linearly. For a quadratic loss function the increase would be quadratic. Thus, a Huber loss reduces the influence of bad leverage points, but does not bound it. For $\beta(H_0)=0$ and for all loss functions, the $L_1$-penalized functionals produce a constantly zero influence function, thus, creating sparseness also under small perturbation from the model. To sum up, a Huber loss function performs better than a quadratic loss, but both cannot bound the influence of bad leverage points. Only sparse LTS and the penalized M-functional with biweight loss are very robust. They are able to bound the impact of observations that lie far away from the model, while observations that closely follow the model get a very high influence.

We simulate the expected values that appear in the influence function (\ref{eq:IFgen}) by Monte Carlo simulation (using $10^5$ replications). Furthermore, Proposition \ref{prop:gen} can actually not be applied as the lasso penalty is not differentiable. However, using either the $tanh$ approximation (\ref{eq:tanh_approx}) or the same approach as in the proof of Lemma \ref{prop:lasso_multi}, one can show that the influence function of these functionals equals zero in case the functional equals zero and (\ref{eq:IFgen}) otherwise.

\begin{figure}
\begin{center}
	\includegraphics[width = \textwidth]{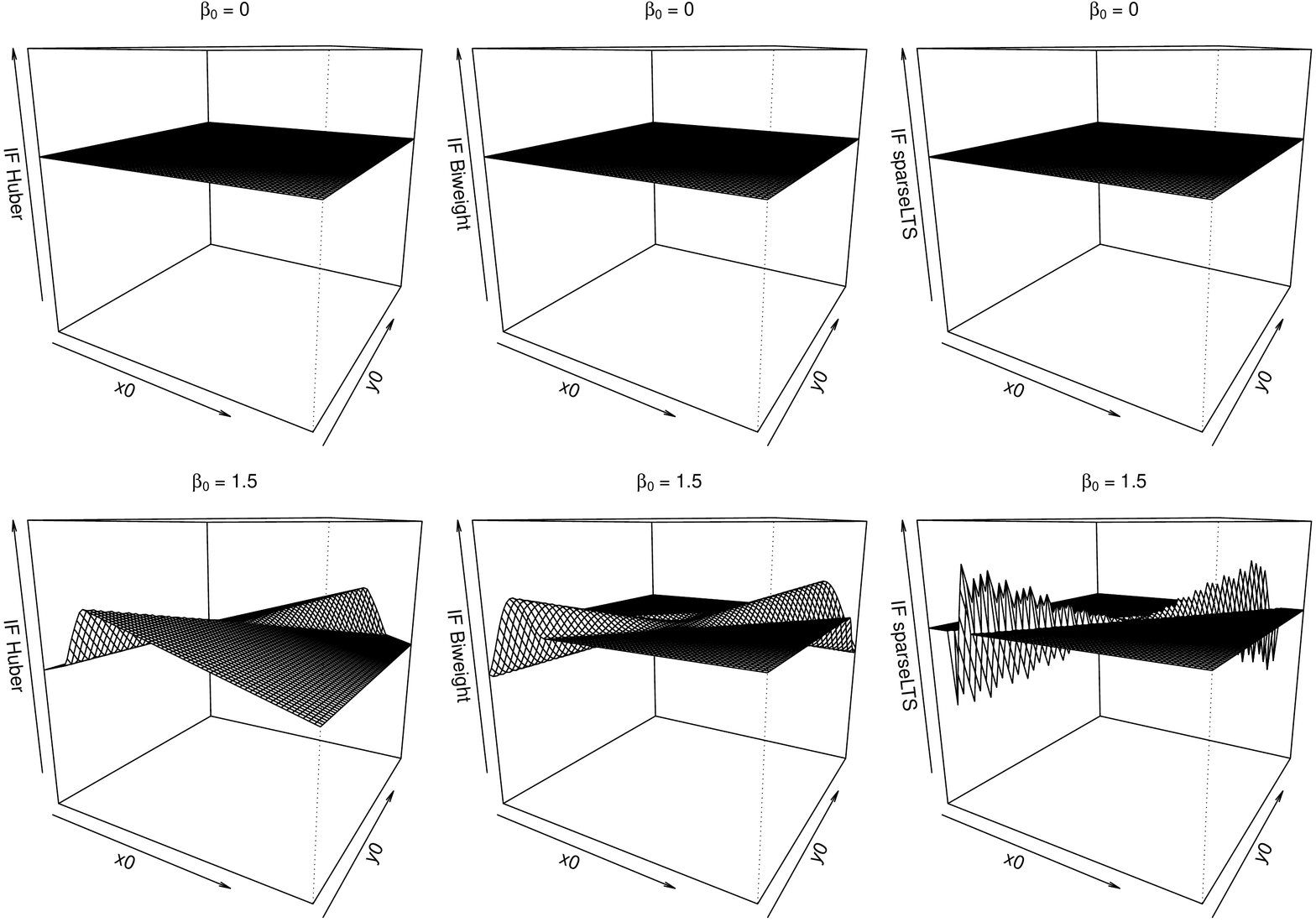}
\caption{Influence functions for different loss functions (Huber, biweight, sparse LTS) and $L_1$-penalty for $\beta_0 = 0$ and $\beta_0 = 1.5$ with $(x_0, y_0)\in[-10,10]^2$ and the vertical axis ranging from $-75$ to $40$}
\label{fig:IF_loss}
\end{center}
\end{figure}


\section{Sensitivity Curves}
\label{sec:NumExp}

To study the robustness of the different penalized M-estimators from Section \ref{sec:Plot_IF} at sample level, we compute sensitivity curves \citep[][]{Maronna}, an empirical version of the influence function. For an estimator $\mathbf{\hat{\beta}}$ and at sample $(X,\mathbf{y})$, it is defined as 
\begin{align*}
SC(\mathbf{x}_0,y_0,\mathbf{\hat{\beta}}) = \frac{\mathbf{\hat{\beta}}(X\cup\{\mathbf{x}_0\}, \mathbf{y}\cup \{y_0\}) - \mathbf{\hat{\beta}}(X, \mathbf{y})}{\frac{1}{n+1}}.
\end{align*}

To compute the penalized estimators, we use the coordinate descent algorithm. As a starting value, we use the least squares estimate for estimators using a quadratic loss, and the robust sparse LTS-estimate for the others. Sparse LTS can be easily and fast computed using the \texttt{sparseLTS} function of the \textsf{R} package \texttt{robustHD}. Furthermore, we divide the argument of the $\rho$-function in (\ref{eq:betaM_data}) by a preliminary scale estimate. For simplicity we use the MAD of the residuals of the initial estimator used in the coordinate descent algorithm.

Figures~\ref{fig:SC_penalty_beta15} and \ref{fig:SC_penalty_beta0} show the sensitivity curves for estimators $\mathbf{\hat{\beta}}$ with quadratic loss function and the different penalties least squares, ridge, lasso and SCAD for parameters $\beta_0=1.5$ and $\beta_0=0$, respectively. We can compare these figures to the theoretical influence functions in Figures~\ref{fig:IF_penalty_beta15} and \ref{fig:IF_penalty_beta0}. Examining Figure~\ref{fig:SC_penalty_beta15}, we see that for $\beta_0=1.5$, the results match the theoretical ones. For $\beta_0 = 0$, see Figure~\ref{fig:SC_penalty_beta0}, the sensitivity curve is again comparable to the influence function. For the lasso and SCAD, small deviations from the constantly zero sensitivity curve can be spotted in the left and right corner. This indicates that the number of observations $n$ is too small to get the same results as at population level for observations $(x_0, y_0)$ that lie far away from the model.

\begin{figure}
\vspace{-0.8cm}
\begin{center}
	\includegraphics[width = 0.88\textwidth]{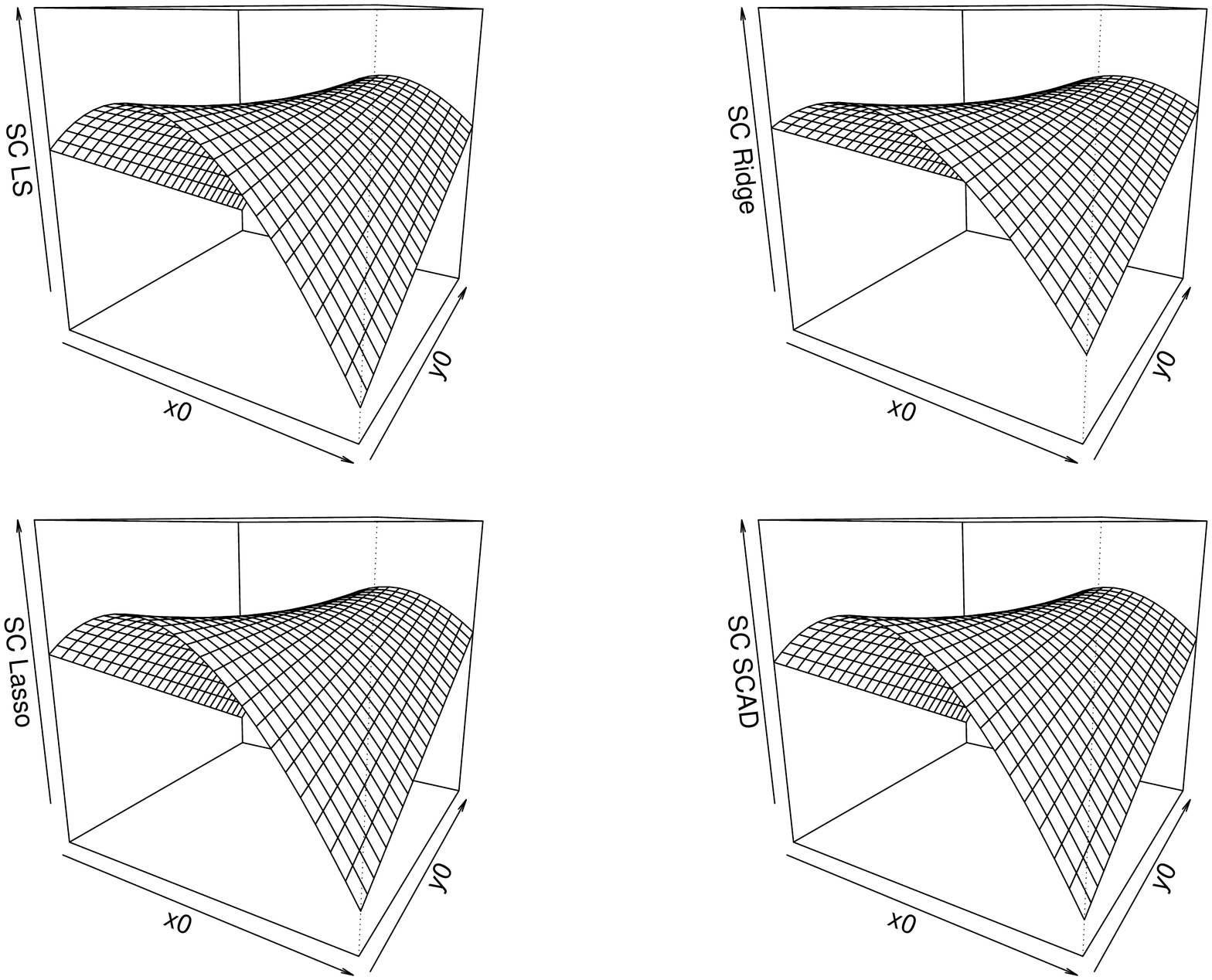}
\caption{Sensitivity curve for different penalty functions (least squares, ridge, lasso and SCAD) for $\beta_0 = 1.5$ with $(x_0, y_0)\in[-10,10]^2$ and the vertical axis ranging from $-250$ to $100$}
\label{fig:SC_penalty_beta15}
\centering
	\includegraphics[width = 0.88\textwidth]{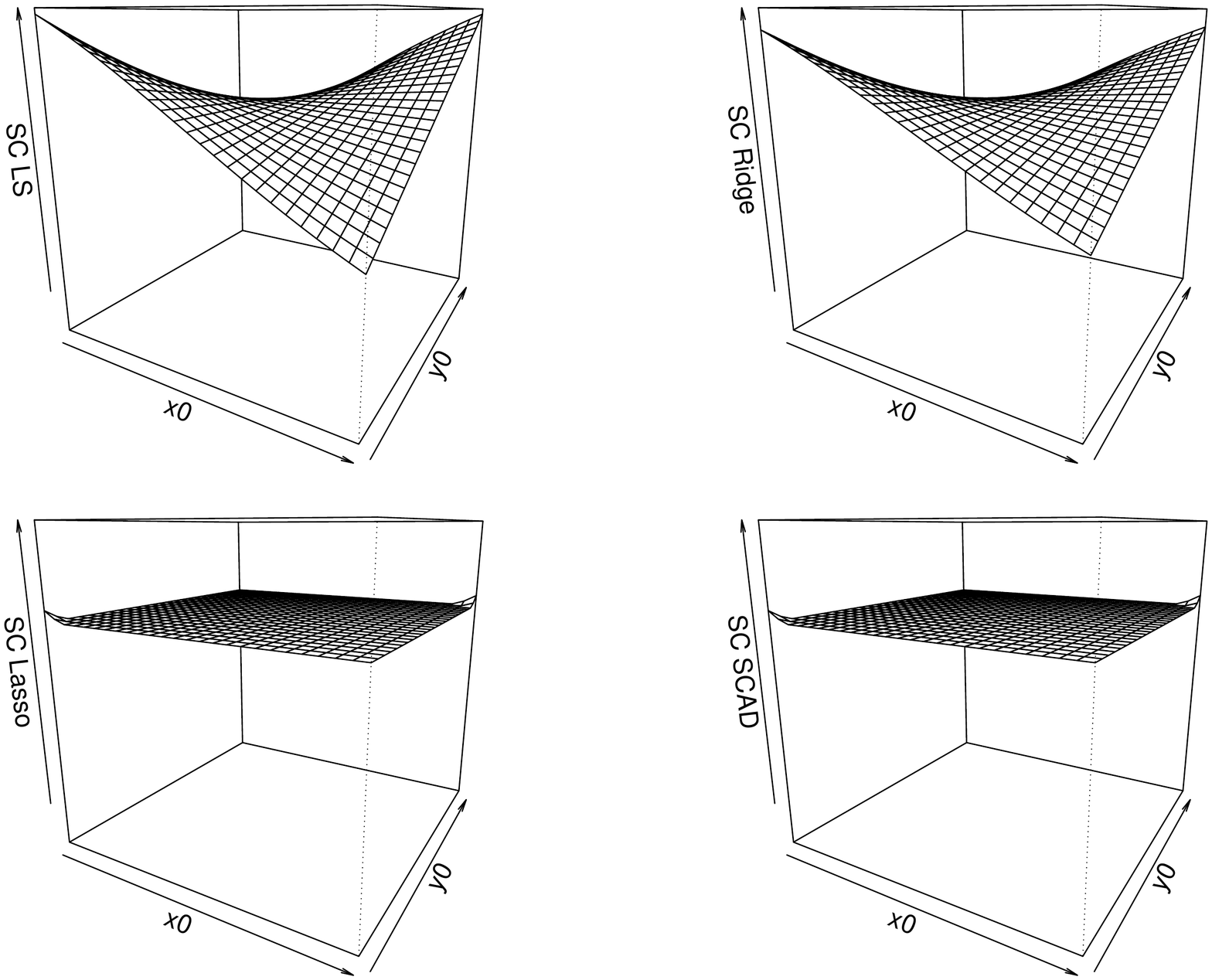}
\caption{Sensitivity curve for different penalty functions (least squares, ridge, lasso and SCAD) for $\beta_0 = 0$ with $(x_0, y_0)\in[-10,10]^2$ and the vertical axis ranging from $-250$ to $100$}
\label{fig:SC_penalty_beta0}
\end{center}
\end{figure}

We also compare the results for estimators using different loss functions. Therefore we look at sparse LTS and the $L_1$-penalized Huber- and biweight-M-estimators, as in Section \ref{sec:Plot_IF}. Their sensitivity curves are plotted in Figure~\ref{fig:SC_loss}. They resemble the shape of the influence functions in Figure~\ref{fig:IF_loss}.

\begin{figure}
\begin{center}
	\includegraphics[width = \textwidth]{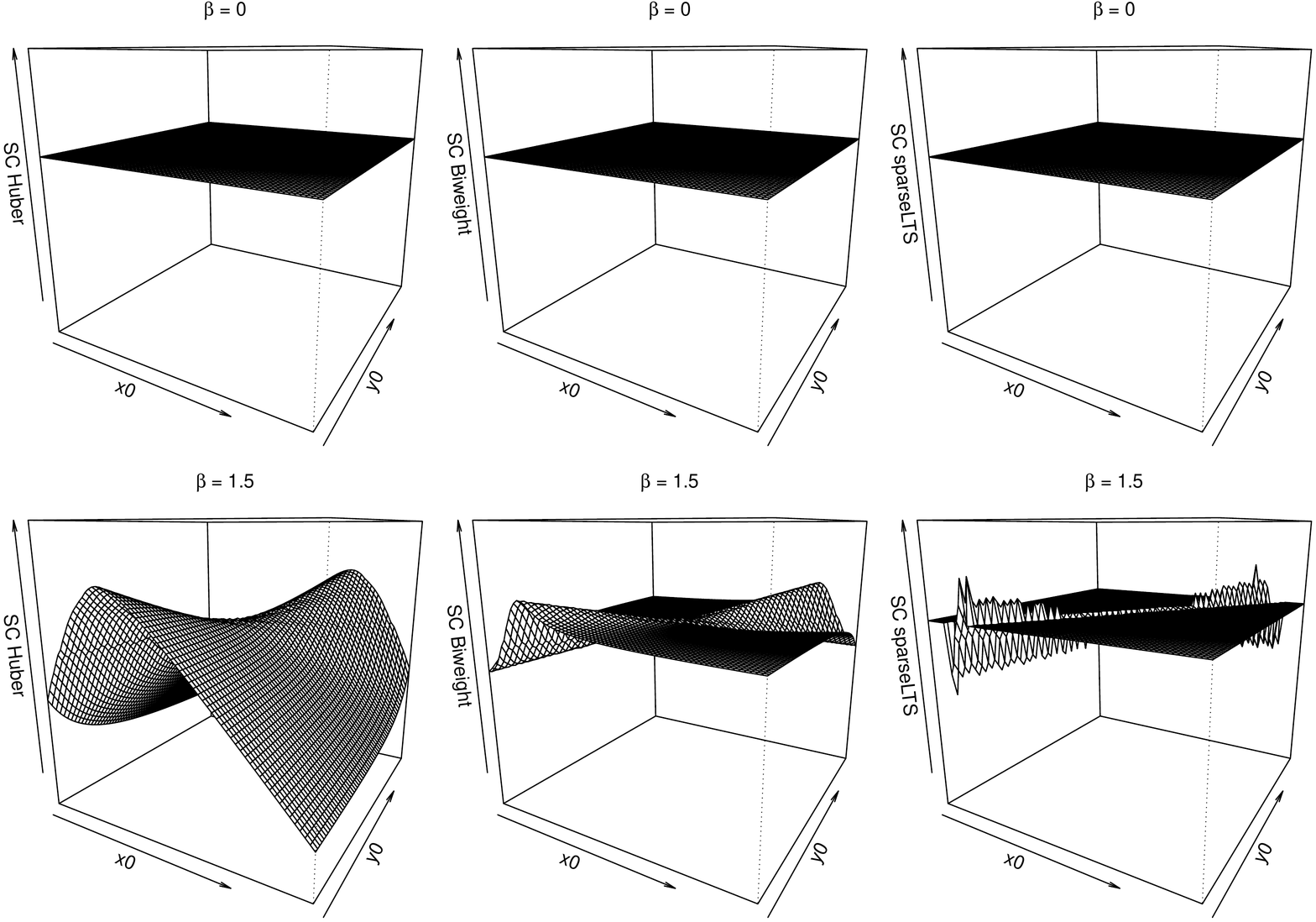}
\caption{Sensitivity curve for different loss functions (Huber, biweight, sparse LTS) and $L_1$-penalty for $\beta_0 = 0$ and $\beta_0 = 1.5$ with $(x_0, y_0)\in[-10,10]^2$ and the vertical axis ranging from $-75$ to $40$}
\label{fig:SC_loss}
\end{center}
\end{figure}	

To conclude, we may say that the sensitivity curves match the corresponding influence functions.


\section{Asymptotic Variance and Mean Squared Error}
\label{sec:ASVMSE}

We can also evaluate the performance of any functional $T$ by the asymptotic variance, given by
\begin{align*}
ASV(T, H) = n\cdot \lim_{n\rightarrow\infty}\Var{\,T_n},
\end{align*}
where the estimator $T_n$ is the functional $T$ evaluated at the empirical distribution. A heuristic formula to compute the asymptotic variance is given by
\begin{align}
ASV(T, H) = \int IF((\mathbf{x}_0,y_0), T, H) \cdot IF((\mathbf{x}_0,y_0), T, H)' \, dH((\mathbf{x}_0,y_0)).
\label{eq:ASV}
\end{align}
For M-functionals with a smooth loss function $\rho$ and smooth penalty $J$, the theory of M-estimators is applicable \citep[e.g.][]{Huber, Hayashi}. For the sparse LTS-estimator, a formal proof of the validity of (\ref{eq:ASV}) is more difficult and we only conjecture its validity. For the unpenalized case a proof can be found in \citep{Hossjer}.

Using formulas of Sections \ref{sec:IF} - \ref{sec:IF_spLTS}, the computation of the integral (\ref{eq:ASV}) is possible using Monte Carlo numerical integration. We present results for simple regression.

Figure~\ref{fig:ASV_uni} shows the asymptotic variance of six different functionals (least squares, lasso, ridge, biweight loss with $L_1$-penalty, Huber loss with $L_1$-penalty, sparse LTS) as a function of $\lambda$ for $\beta_0 = 1.5$. As the asymptotic variance of least squares is constantly one for any value $\lambda$ and $\beta_0$, it is used as a reference point in all four panels. All sparse functionals show a jump to zero in their asymptotic variance after having increased quickly to their maximum. This is due to parameters estimated exactly zero, for values of $\lambda$ sufficiently large. In the left upper panel, the asymptotic variance of ridge is added. It is smaller than the asymptotic variance of least squares and decreases monotonously to zero. Generally, for the optimal $\lambda$, least squares has high asymptotic variance, ridge a reduced one. The smallest asymptotic variance can be achieved by the sparse functionals. But they can also get considerably high values for bad choices of $\lambda$. We omit the plots for $\beta_0=0$ because the asymptotic variance of ridge behaves similarly as in Figure~\ref{fig:ASV_uni} and the asymptotic variance of the other, sparse functionals is constantly zero.

\begin{figure}
\begin{center}
	\includegraphics[width = \textwidth]{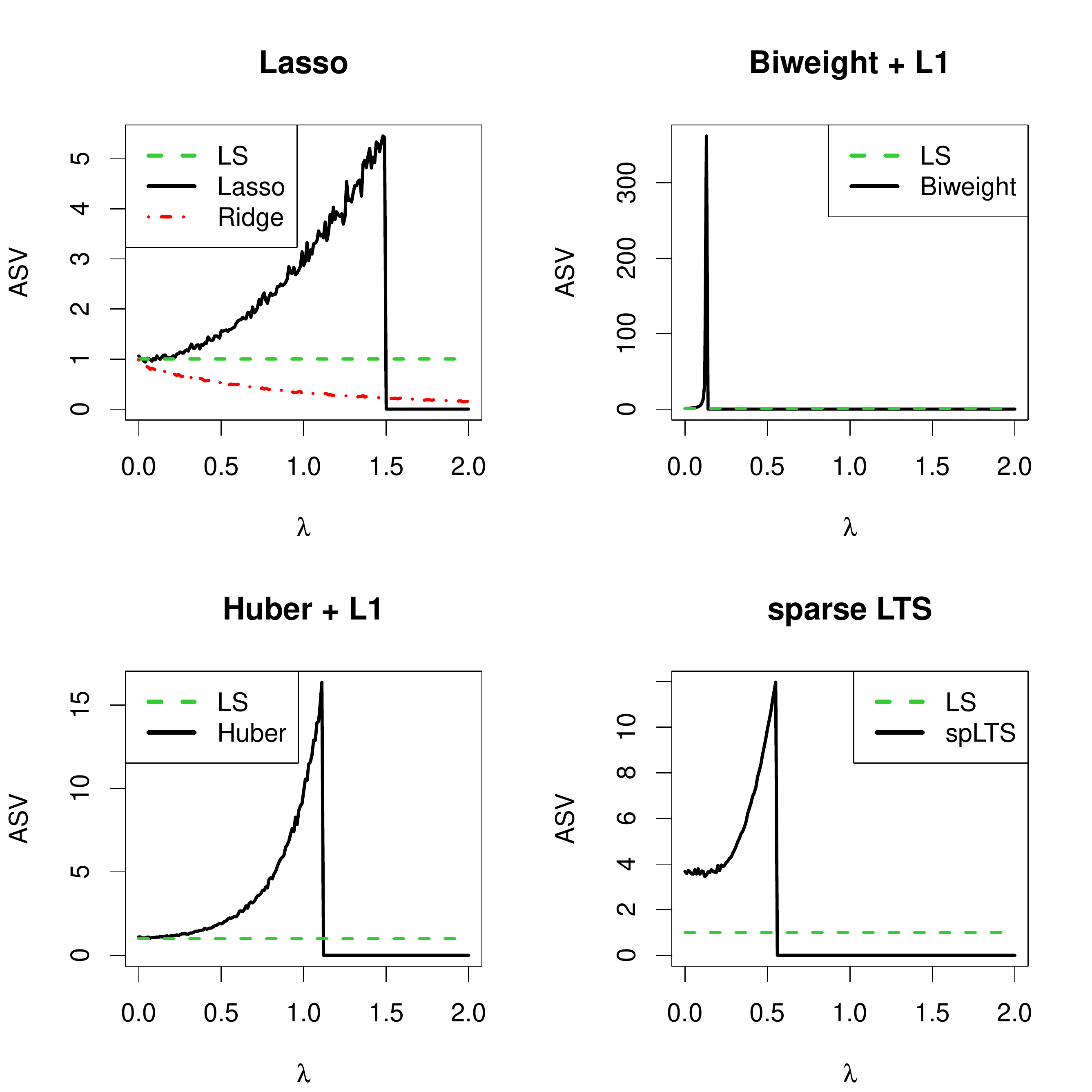}
\caption{Asymptotic variance of various functionals for $\beta_0=1.5$}
\label{fig:ASV_uni}
\end{center}
\end{figure}

In general, robust functionals have a bias (see Section \ref{sec:Bias}). Hence, considering only the asymptotic variance is not sufficient to evaluate the precision of functionals. A more informative measure is the Mean Squared Error (MSE) as it takes bias and variance into account
\begin{align}
MSE(T, H) = \frac{1}{n} ASV(T, H) + \Bias(T, H) \Bias(T, H)'.
\label{eq:MSE}
\end{align}
Figure~\ref{fig:MSE_uni} displays $\MSE$ as a function of $n$ for $\beta_0=0.05$ and $1.5$, $\lambda = 0.1$ is fixed. We only present results for simple regression as they resemble the component-wise results in multiple regression.

Looking at Figure~\ref{fig:MSE_uni}, the MSE of least squares is the same in both panels as least squares has no bias and its asymptotic variance does not depend on $\beta_0$. It decreases monotonously from one to zero. The MSEs of the other functionals are also monotonously decreasing, but towards their bias. For $\beta_0 = 0.05$, MSE of ridge is slightly lower than that of least squares. The MSEs of the sparse functionals are constant and equal to their squared bias (i.e. $\beta_0^2$ as the estimate equals zero). For $\beta_0=1.5$, MSE of biweight is largest, MSE of sparse LTS is slightly larger than ridge and MSE of the lasso and Huber is similar to least squares, which is the lowest. We again do not show results for $\beta_0=0$ because then no functional has a bias, and we would only compare the asymptotic variances.

\begin{figure}
\begin{center}
	\includegraphics[width = 0.9\textwidth]{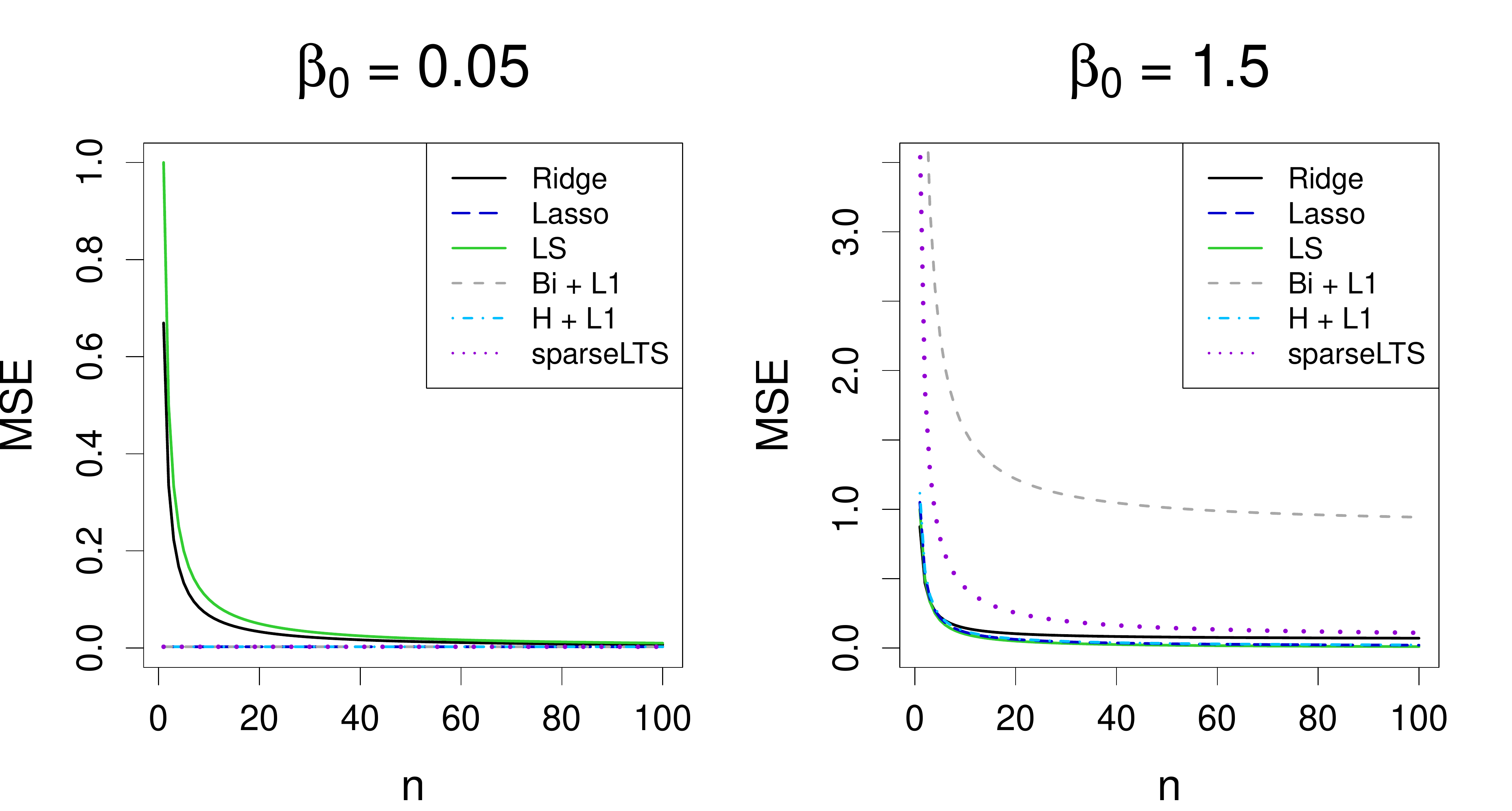}
\caption{Mean squared error of various functionals ($\lambda = 0.1$ fixed)}
\label{fig:MSE_uni}
\end{center}
\end{figure}	

We also show the match at population and sample level for the MSE. For any estimator $\hat{\beta}$ computed for $r=1,\ldots,R$ samples, an estimator for the mean squared error (\ref{eq:MSE}) is 
\begin{align*}
\widehat{MSE}(\hat{\beta}) = \frac{1}{R}\sum_{r=1}^R (\hat{\beta}_r - \beta_0)^2.
\end{align*}
For the six functionals (least squares, ridge, lasso, biweight-M wih $L_1$-penalty, Huber-M with $L_1$-penalty and sparse LTS) used in this section, Figures~\ref{fig:MSE_comp2} and \ref{fig:MSE_comp3} illustrate the good convergence of $n\cdot\widehat{MSE}(\hat{\beta})$ to $n\cdot MSE(\beta_0, H_0)$ for $\beta_0 = 0.05$ and $1.5$, respectively.

\begin{figure}
\begin{center}
	\includegraphics[width = \textwidth]{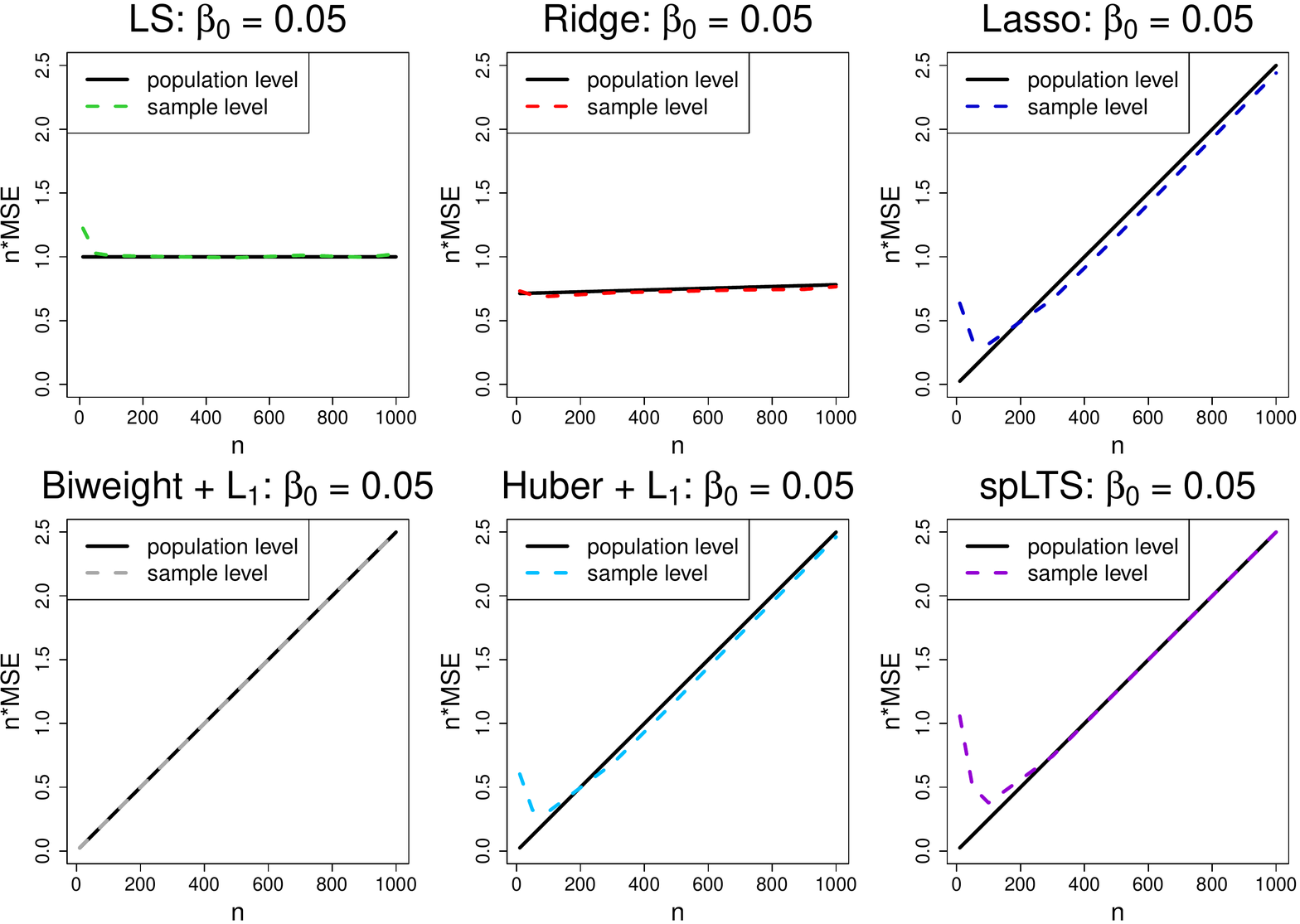}
\caption{Convergence of $\widehat{MSE}(\hat{\beta})$ to $MSE(\beta_0, H_0)$ for different functionals with $\beta_0 = 0.05$}
\label{fig:MSE_comp2}
\end{center}
\end{figure}	

\begin{figure}
\begin{center}
	\includegraphics[width = \textwidth]{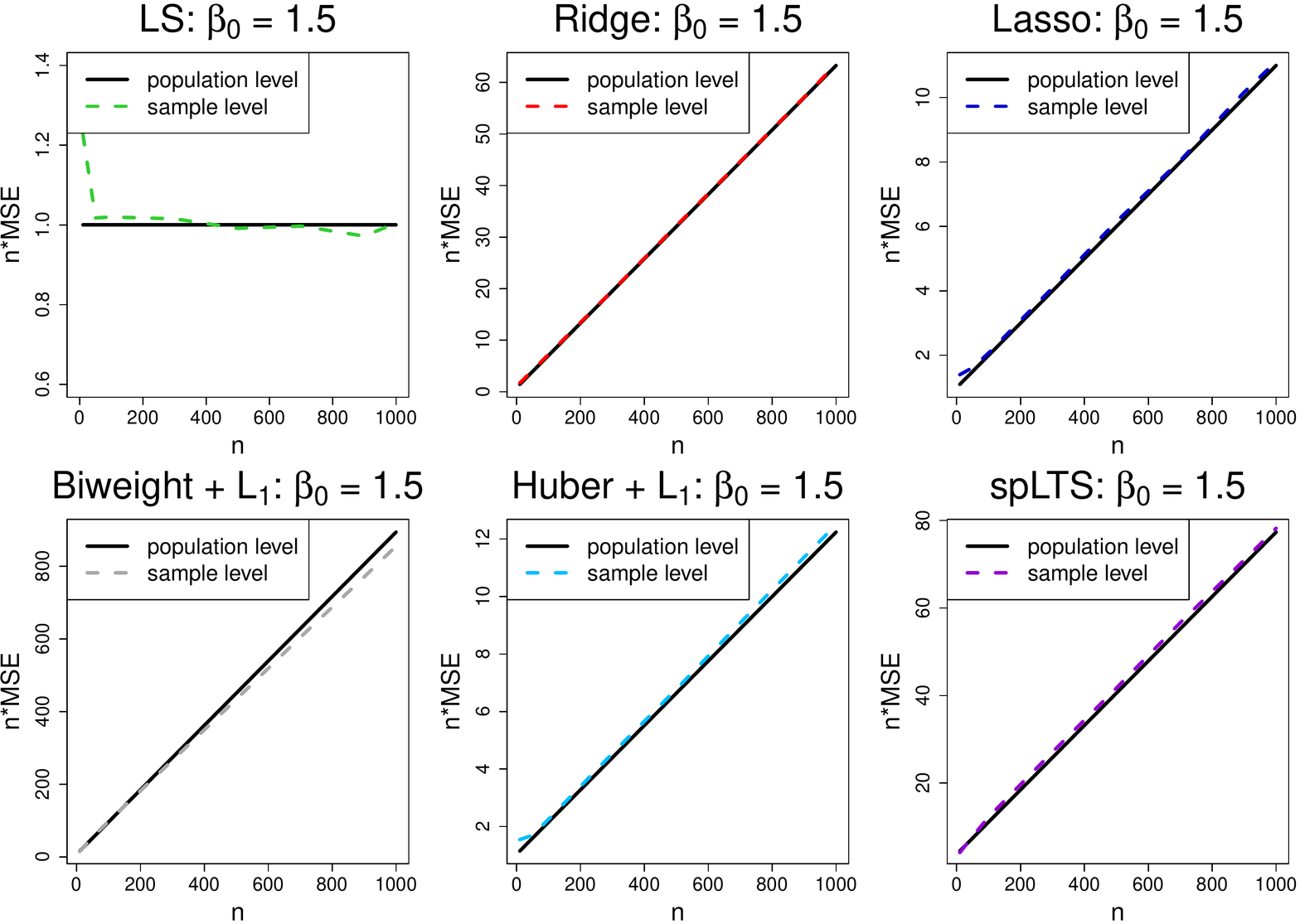}
\caption{Convergence of $\widehat{MSE}(\hat{\beta})$ to $MSE(\beta_0, H_0)$ for different functionals with $\beta_0 = 1.5$}
\label{fig:MSE_comp3}
\end{center}
\end{figure}


\section{Conclusion}
\label{sec:Conc}

In this paper we computed influence functions of  penalized regression estimators, more precisely for penalized M-functionals. From the derivation of the influence function, we concluded that only functionals with a bounded loss function (biweight, sparse LTS) achieve robustness against leverage points, while a Huber loss can deal with vertical outliers. Looking at the MSE, sparse LTS is preferred in case of bad leverage points and the $L_1$-penalized Huber M-estimator in case there are only vertical outliers. 

Apart from considering the influence function, a suitable estimator is often also chosen with respect to its breakdown point \citep[see for example][]{Maronna}. This second important property in robust analysis gives the maximum fraction of outliers that a method can deal with. While it has already been computed for sparse LTS \citep[][]{Alfons}, it would also be worth deriving it for the other robust penalized M-functionals mentioned in this paper.

As any study, also this one is subject to some limitations. First of all, we assumed in our derivations the penalty parameter $\lambda$ to be fixed. However, in practice it is often chosen with a data-driven approach. Thus, contamination in the data might also have an effect on the estimation through the choice of the penalty parameter. Investigation of this effect is left for further research.

Another limitation is that the values of the tuning constants in the loss functions of the M-estimators were selected to achieve a given efficiency in the non penalized case. One could imagine to select the $\lambda$ parameter simultaneously with the other tuning constants. 

Finally, in the theoretical derivations (but not at the sample level) we implicitly assume the scale of the error terms to be fixed, in order to keep the calculations feasible.  While the results obtained for the lasso, the ridge and the sparse LTS functional do not rely on that assumption, the results for biweight and Huber loss do.


\newpage
\appendix
\section*{APPENDIX - Proofs}
\begin{proof}[Proof of Equation \ref{eq:lassoFunct}]
Recall that we are in the case $p=1$. For any joint distribution $(x,y)\sim H$ with $\beta_{LASSO}(H)\neq 0$, minimizing the objective function in (\ref{eq:lassoMulti}) and solving the resulting first-order condition (FOC) for $\beta_{LASSO}(H)$ yields
\begin{align}
\beta_{LASSO}(H) = \beta_{LS}(H)-\frac{\lambda}{\mathbb{E}_H[x^2]}\sign(\beta_{LASSO}(H)).
\label{eq:lasso1}
\end{align}

We will now consider two different cases. First we consider the case that the lasso functional is not zero at distribution $H$. We will show that it then always has to have the same sign as the least squares functional $\beta_{LS}(H)$. We start with assuming $\sign(\beta_{LASSO}(H))\neq\sign(\beta_{LS}(H))$ and show that this will lead to a contradiction. In this case $\beta_{LS}(H)=0$ is not possible for the following reason. If $\beta_{LS}(H)=0$, then $\beta=0$ minimizes the residual sum of squares. Furthermore, the minimum of the penalty function is attained at $\beta=0$. Hence, $\beta=0$ would not only minimize the residual sum of squares, but also the penalized objective function, if $\beta_{LS}(H) = 0$. Hence, the lasso functional would also be zero, which we do not consider in this first case. Thus, take $\beta_{LS}(H)>0$. From our assumption it would follow that $\sign(\beta_{LASSO}(H))=-1$ (as $\beta_{LASSO}(H)=0$ is considered only in the next paragraph) and together with the FOC this would yield the contradiction $0 > \beta_{LASSO}(H)=\beta_{LS}(H)+\lambda/\mathbb{E}_H[x^2]>\beta_{LS}(H)>0$. Analogous for $\beta_{LS}(H)<0$. Hence, for $\beta_{LASSO}(H)\neq 0$ the sign of the lasso and the least squares functional are always equal.

Let's now consider the case where the lasso functional is zero at the distribution $H$. The FOC then makes use of the concept of subdifferentials \citep[][]{Bertsekas} and can be written as $|\beta_{LS}(H)|\leq\lambda/\mathbb{E}_H[x^2]$. On the other hand, if $|\beta_{LS}(H)|\leq\lambda/\mathbb{E}_H[x^2]$ assuming $\beta_{LASSO}(H)\neq 0$ leads to a contradiction since Equation (\ref{eq:lasso1}) would imply that $\sign(\beta_{LASSO}(H)) = - \sign(\beta_{LASSO}(H))$. Thus, the lasso functional equals zero if and only if $|\beta_{LS}(H)|\leq\lambda/\mathbb{E}_H[x^2]$.
Therefore the lasso functional for simple regression is (\ref{eq:lassoFunct}).
\end{proof}

\begin{proof}[Proof of Lemma \ref{lemma:spLTS_uni}]
As $x\sim\mathcal{N}(0,\Sigma)$ and $e\sim\mathcal{N}(0,\sigma^2)$ are independent, $y-x\beta$ is normally distributed $y-x\beta\sim\mathcal{N}(0, \sigma^2+(\beta_0-\beta)^2\Sigma)$ for any $\beta\in\mathbb{R}$. Defining $\sigma^2(\beta) := \sigma^2+(\beta_0-\beta)^2\Sigma)$ we find $q_\beta = \Phi^{-1}(\frac{\alpha + 1}{2})\sigma(\beta)$. We also introduce $q_\alpha = \Phi^{-1}(\frac{\alpha+1}{2})$. With this we can rewrite the expected value of the objective function (\ref{eq:spLTS})
\begin{align}
\mathbb{E}_{H_0}[(y- x\beta)^2 I_{[|y-x\beta|\leq q_\beta]}] &= \sigma^2(\beta) \mathbb{E}_{H_0}[\frac{(y- x\beta)^2}{\sigma^2(\beta)} I_{[\frac{|y- x\beta|}{\sigma(\beta)}\leq q_\alpha]}]\notag\\
&=\sigma^2(\beta) \mathbb{E}_Z[Z^2 I_{[|Z|\leq q_\alpha]}]\qquad\text{ with } Z\sim\mathcal{N}(0,1)\notag\\
&=\sigma^2(\beta)(-2 q_\alpha \phi(q_\alpha) +\alpha).
\label{eq:spLTSaux2}
\end{align}
Denoting $c_1:=\alpha-2 q_\alpha \phi(q_\alpha)$, we can say that
\begin{align*}
\beta_{spLTS}(H_0) = \argmin_{\beta\in\mathbb{R}} c_1 \sigma^2(\beta) + \alpha \lambda |\beta|.
\end{align*}
Separating into $\beta\geq0$ and $\beta\leq0$, differentiating w.r.t. $\beta$ and setting the result to $0$ gives Equation (\ref{eq:spLTS_uni}).
\end{proof}

\begin{proof}[Proof of Proposition \ref{prop:gen}]
The objective function (\ref{eq:betaM}) is minimized by solving the first-order condition (FOC), the derivative of the objective function set zero. At the contaminated model with distribution $H_\epsilon:=(1-\epsilon)H_0+\epsilon\,\delta_{(\mathbf{x}_0,y_0)}$ this yields
\begin{align*}
-\mathbb{E}_{H_\epsilon}[\psi(y-\mathbf{x}'\mathbf{\beta}_M(H_\epsilon))\mathbf{x}] + 2\lambda J'(\mathbf{\beta}_M(H_\epsilon))=0.
\end{align*}
Here $J'(\mathbf{\beta}_M(H_\epsilon))$ is used as an abbreviation for $(J'(\beta_1(H_\epsilon)),\ldots,J'(\beta_p(H_\epsilon)))'$  and $\delta_{(\mathbf{x}_0,y_0)}$ denotes the point mass distribution at $(\mathbf{x}_0,y_0)$.

Using the definition of the contaminated distribution $H_\epsilon$, the FOC becomes
\begin{align*}
-(1-\epsilon)\mathbb{E}_{H_0}[\psi(y-\mathbf{x}'\mathbf{\beta}_M(H_\epsilon))\mathbf{x}] - \epsilon \psi(y_0-\mathbf{x}_0'\mathbf{\beta}_M(H_\epsilon))\mathbf{x}_0+ 2\lambda J'(\mathbf{\beta}_M(H_\epsilon))=0.
\end{align*}
Derivation with respect to $\epsilon$  yields
\begin{align*}
&\mathbb{E}_{H_0}[\psi(y-\mathbf{x}'\mathbf{\beta}_M(H_\epsilon))\mathbf{x}]- (1-\epsilon)\mathbb{E}_{H_0}[\psi'(y-\mathbf{x}'\mathbf{\beta}_M(H_\epsilon))\mathbf{x}(-\mathbf{x}'\frac{\partial}{\partial\epsilon} \mathbf{\beta}_M(H_\epsilon))] \\
&\qquad - \psi(y_0-\mathbf{x}_0'\mathbf{\beta}_M(H_\epsilon))\mathbf{x}_0 - \epsilon \psi'(y_0-\mathbf{x}_0'\mathbf{\beta}_M(H_\epsilon))\mathbf{x}_0(-\mathbf{x}_0'\frac{\partial}{\partial\epsilon} \mathbf{\beta}_M(H_\epsilon))\\
&\qquad + 2\lambda \diag(J''(\mathbf{\beta}_M(H_\epsilon)))\frac{\partial}{\partial\epsilon}\mathbf{\beta}_M(H_\epsilon) = 0,
\end{align*}
where $\diag(J''(\mathbf{\beta}_M(H_\epsilon)))$ denotes the diagonal matrix with entries \\
$(J''((\beta_M(H_\epsilon))_1),\ldots,J''((\beta_M(H_\epsilon))_p))$ in the main diagonal.

Since $\frac{\partial}{\partial\epsilon}\big[\mathbf{\beta}_M(H_\epsilon)\big]\big|_{\epsilon=0} = IF((\mathbf{x}_0,y_0), \mathbf{\beta}_M, H_0)$, 
\begin{align}
&\mathbb{E}_{H_0}[\psi(y-\mathbf{x}'\mathbf{\beta}_M(H_0))\mathbf{x}] + \mathbb{E}_{H_0}[\psi'(y-\mathbf{x}'\mathbf{\beta}_M(H_0))\mathbf{x}\mathbf{x}']\cdot IF((\mathbf{x}_0,y_0),\mathbf{\beta}_M,H_0) \\
&\qquad- \psi(y_0-\mathbf{x}_0'\mathbf{\beta}_M(H_0))\mathbf{x}_0 + 2\lambda \diag(J''(\mathbf{\beta}_M(H_0))) \cdot IF((\mathbf{x}_0,y_0), \mathbf{\beta}_M, H_0) = 0,
\label{eq:betaMaux}
\end{align}
Solving (\ref{eq:betaMaux}) for $IF((\mathbf{x}_0,y_0), \mathbf{\beta}_M, H_0)$, gives Equation (\ref{eq:IFgen}).
\end{proof}

\begin{proof}[Proof of Lemma \ref{lemma:lasso_uni}]
Using the explicit definition of the lasso functional (\ref{eq:lassoFunct}), its influence function can be computed directly. Thus, we differentiate the functional at the contaminated model $H_\epsilon=(1-\epsilon)H_0+\epsilon\delta_{(x_0,y_0)}$ with respect to $\epsilon$ and take the limit of $\epsilon$ approaching $0$
\begin{align*}
IF(&(x_0,y_0), \beta_{LASSO},H_0)=\\
&=\frac{\partial}{\partial\epsilon}\left[\sign((1-\epsilon)\mathbb{E}_{H_0}[xy]+\epsilon x_0y_0)\left(\left|\frac{(1-\epsilon)\mathbb{E}_{H_0}[xy]+\epsilon x_0y_0}{(1-\epsilon)\mathbb{E}_{H_0}[x^2]+\epsilon x_0^2}\right|-\frac{\lambda}{(1-\epsilon)\mathbb{E}_{H_0}[x^2]+\epsilon x_0^2}\right)_+\right]\bigg|_{\epsilon=0}\\
&=\frac{\partial}{\partial\epsilon}\left[\sign((1-\epsilon)\mathbb{E}_{H_0}[xy]+\epsilon x_0y_0)\right]\bigg|_{\epsilon=0}\left(\left|\frac{\mathbb{E}_{H_0}[xy]}{\mathbb{E}_{H_0}[x^2]}\right|-\frac{\lambda}{\mathbb{E}_{H_0}[x^2]}\right)_{+}+\\
&\quad+\sign(\mathbb{E}_{H_0}[xy])\frac{\partial}{\partial\epsilon}\left[\left(\left|\frac{(1-\epsilon)\mathbb{E}_{H_0}[xy]+\epsilon x_0y_0}{(1-\epsilon)\mathbb{E}_{H_0}[x^2]+\epsilon x_0^2}\right|-\frac{\lambda}{(1-\epsilon)\mathbb{E}_{H_0}[x^2]+\epsilon x_0^2}\right)_+\right]\bigg|_{\epsilon=0}.
\end{align*}
While the derivative in the first summand equals zero almost everywhere, the derivative occurring in the second summand has to consider two cases separately. Using the fact that $\mathbb{E}_{H_0}[xy]/\mathbb{E}_{H_0}[x^2] = \beta_{LS}(H_0) = \beta_0$, we get
\begin{align*}
\frac{\partial}{\partial\epsilon}&\left[\left(\left|\frac{(1-\epsilon)\mathbb{E}_{H_0}[xy]+\epsilon x_0y_0}{(1-\epsilon)\mathbb{E}_{H_0}[x^2]+\epsilon x_0^2}\right|-\frac{\lambda}{(1-\epsilon)\mathbb{E}_{H_0}[x^2]+\epsilon x_0^2}\right)_+\right]\bigg|_{\epsilon=0}=\\
&=\begin{cases}
0 &\hspace{-3cm}\text{ if }-\frac{\lambda}{\mathbb{E}_{H_0}[x^2]}\leq \beta_0<\frac{\lambda}{\mathbb{E}_{H_0}[x^2]}\\
\sign\left(\frac{\mathbb{E}_{H_0}[xy]}{\mathbb{E}_{H_0}[x^2]}\right)\left(\frac{(-\mathbb{E}_{H_0}[xy]+x_0y_0)\mathbb{E}_{H_0}[x^2]-\mathbb{E}_{H_0}[xy](-\mathbb{E}_{H_0}[x^2]+x_0^2)}{\left(\mathbb{E}_{H_0}[x^2]\right)^2}\right) + \frac{\lambda\left(-\mathbb{E}_{H_0}[x^2]
+x_0^2\right)}{\left(\mathbb{E}_{H_0}[x^2]\right)^2} &\text{ otherwise}
\end{cases}\\
&= \begin{cases}
0 &\text{ if }-\frac{\lambda}{\mathbb{E}_{H_0}[x^2]}\leq \beta_0 <\frac{\lambda}{\mathbb{E}_{H_0}[x^2]}\\
\sign(\beta_0)\left(\frac{x_0(y_0-\beta_0 x_0)}{\mathbb{E}_{H_0}[x^2]}\right)-\lambda\frac{\mathbb{E}_{H_0}[x^2]-x_0^2}{\left(\mathbb{E}_{H_0}[x^2]\right)^2} &\text{ otherwise}.
\end{cases}
\end{align*}
Thus, almost everywhere the influence function equals (\ref{eq:IF_lasso_uni}).
\end{proof}

\begin{proof}[Proof of Lemma \ref{lemma:lasso_shooting}]
Differentiating the lasso functional of the coordinate descent algorithm
\begin{align*}
\beta_j^{cd}(H)= \sign\left(\mathbb{E}_H\left[x_j (y-{\mathbf{x}^{(j)}}'\mathbf{\beta^\ast}^{(j)})\right]\right)\left(\left|\frac{\mathbb{E}_H\left[x_j (y-{\mathbf{x}^{(j)}}'\mathbf{\beta^\ast}^{(j)})\right]}{\mathbb{E}_H[x_j^2]}\right|-\frac{\lambda}{\mathbb{E}_H[x_j^2]}\right)_+
\end{align*}
for the contaminated model $(\mathbf{x},y)\sim H_\epsilon=(1-\epsilon)H_0+\epsilon\delta_{(\mathbf{x}_0,y_0)}$ yields
\begin{align}
IF(&(\mathbf{x}_0,y_0), \beta_j^{cd}, H_0, \mathbf{\beta^\ast})=\notag\\
&=\frac{\partial}{\partial\epsilon}\left[\sign\left(\mathbb{E}_{H_\epsilon}\left[x_j \left(y-{\mathbf{x}^{(j)}}'\mathbf{\beta^\ast}^{(j)}(\epsilon)\right)\right]\right)\right]\bigg|_{\epsilon=0}\left(\left|\frac{\mathbb{E}_{H_0}[x_j (y-{\mathbf{x}^{(j)}}'\mathbf{\beta^\ast}^{(j)}]}{\mathbb{E}_{H_0}[x_j^2]}\right|-\frac{\lambda}{\mathbb{E}_{H_0}[x_j^2]}\right)_{+}+\notag\\
&\quad+\sign\left(\mathbb{E}_{H_0}\left[x_j \left(y-{\mathbf{x}^{(j)}}'\mathbf{\beta^\ast}^{(j)}\right)\right]\right)\frac{\partial}{\partial\epsilon}\left[\left(\left|\frac{\mathbb{E}_{H_\epsilon}[x_j (y-{\mathbf{x}^{(j)}}'\mathbf{\beta^\ast}^{(j)}(\epsilon))]}{\mathbb{E}_{H_\epsilon}[x_j^2]}\right|-\frac{\lambda}{\mathbb{E}_{H_\epsilon}[x_j^2]}\right)_+\right]\bigg|_{\epsilon=0}.\label{eq:shoot_aux1}
\end{align}
Note that the fixed values $\mathbf{\beta^\ast}(\epsilon)$ depend on $\epsilon$, as they may depend on the data, e.g. if they are the values of a previous coordinate descent loop. $\mathbf{\beta^\ast}^{(j)}$ is used as an abbreviation for $\mathbf{\beta^\ast}^{(j)}(0)$ and $IF((\mathbf{x}_0,y_0),\mathbf{\beta^\ast}^{(j)},H_0)$ is shortened to $IF(\mathbf{\beta^\ast}^{(j)})$.

The derivative of the sign-function equals zero almost everywhere. For the derivation of the positive part function two different cases have to be considered
\begin{align}
&\hspace{-1.5cm}\frac{\partial}{\partial\epsilon}\left[\left(\left|\frac{(1-\epsilon)\mathbb{E}_{H_0}[x_j \left(y-{\mathbf{x}^{(j)}}'\mathbf{\beta^\ast}^{(j)}(\epsilon)\right)]+\epsilon (\mathbf{x}_0)_j \left(y_0-{\mathbf{x}_0^{(j)}}'\mathbf{\beta^\ast}^{(j)}(\epsilon)\right)}{(1-\epsilon)\mathbb{E}_{H_0}[x_
j^2]+\epsilon (\mathbf{x}_0)_j^2}\right|-\frac{\lambda}{(1-\epsilon)\mathbb{E}_{H_0}[x_j^2]+\epsilon (\mathbf{x}_0)_j^2}\right)_+\right]\bigg|_{\epsilon=0}=\notag\\
&=\begin{cases}
0\hspace{9.5cm}  \text{ if }\left|\frac{\mathbb{E}_{H_0}[x_j(y-{\mathbf{x}^{(j)}}'\mathbf{\beta^\ast}^{(j)})]}{\mathbb{E}_{H_0}[x_j^2]}\right| < \frac{\lambda}{\mathbb{E}_{H_0}[x_j^2]}\\
\sign\hspace{-0.5cm}\left(\frac{\mathbb{E}_{H_0}[x_j\left(y-{\mathbf{x}^{(j)}}'\mathbf{\beta^\ast}^{(j)}\right)]}{\mathbb{E}_{H_0}[x_j^2]}\right)\hspace{-.2cm}\bigg(\frac{(-\mathbb{E}_{H_0} [x_j\left(y - {\mathbf{x}^{(j)}}'\mathbf{\beta^\ast}^{(j)}\right)] + \left(-\mathbb{E}_{H_0}[x_j{\mathbf{x}^{(j)}}'IF(\mathbf{\beta^\ast}^{(j)})]\right) + (\mathbf{x}_0)_j\left(y_0-{\mathbf{x}_0^{(j)}}'\mathbf{\beta^\ast}^{(j)}\right))\mathbb{E}_{H_0}[x_j^2]}{\left(\mathbb{E}_{H_0}[x_j^2]\right)^2}\\
\hspace{4cm} +\,\frac{-\mathbb{E}_{H_0}[x_j\left(y-{\mathbf{x}^{(j)}}'\mathbf{\beta^\ast}^{(j)}\right)](-\mathbb{E}_{H_0}[x_j^2]+(\mathbf{x}_0)_j^2)}{\left(\mathbb{E}_{H_0}[x_j^2]\right)^2}\bigg) - \frac{-\lambda\left(-\mathbb{E}_{H_0}[x_j^2]
+(\mathbf{x}_0)_j^2\right)}{\left(\mathbb{E}_{H_0}[x_j^2]\right)^2} \text{   otherwise}
\end{cases}\notag\\
&=\begin{cases}
0 \hspace{9.5cm} \text{ if }\left|\mathbb{E}_{H_0}[x_j(y-{\mathbf{x}^{(j)}}'\mathbf{\beta^\ast}^{(j)})]\right| < \lambda\\
\sign(\mathbb{E}_{H_0}[x_j(y-{\mathbf{x}^{(j)}}'\mathbf{\beta^\ast}^{(j)})])\hspace{-0.2cm} \left(\frac{-\mathbb{E}_{H_0}[x_j{\mathbf{x}^{(j)}}'IF(\mathbf{\beta^\ast}^{(j)})] + (\mathbf{x}_0)_j\left(y_0-{\mathbf{x}_0^{(j)}}'\mathbf{\beta^\ast}^{(j)}\right)}{\mathbb{E}_{H_0}[x_j^2]}- \frac{\frac{\mathbb{E}_{H_0}[x_j\left(y-{\mathbf{x}^{(j)}}'\mathbf{\beta^\ast}^{(j)}\right)]}{\mathbb{E}_{H_0}[x_j^2]}(\mathbf{x}_0)_j^2}{\mathbb{E}_{H_0}[x_j^2]}\right)\notag\\
\hspace{11cm}-\lambda\frac{\mathbb{E}_{H_0}[x_j^2]-(\mathbf{x}_0)_j^2}{\left(\mathbb{E}_{H_0}[x_j^2]\right)^2} \text{   otherwise}.
\end{cases}\label{eq:shoot_aux2}\\
\end{align}
Using the result of Equation (\ref{eq:shoot_aux2}) in (\ref{eq:shoot_aux1}) and denoting $\tilde{y}^{(j)}:= y - {\mathbf{x}^{(j)}}'{\mathbf{\beta}^\ast}^{(j)}$ yields influence function (\ref{eq:IF_shoot}).
\end{proof}

\begin{proof}[Proof of Proposition \ref{prop:lasso_multi}]
W.l.o.g. $\mathbf{\beta}_{LASSO}=(\mathbf{\tilde{\beta}},0,\ldots,0)'$ with $\mathbf{\tilde{\beta}}\in\mathbb{R}^k$ and $\tilde{\beta}_j\neq0\,\forall j=1,\ldots,k$. At first, we only consider variables $j=1,\ldots,k$. For them, the first-order condition (FOC) for finding the minimum of (\ref{eq:lassoMulti}) yields
\begin{align*}
\left(-2\mathbb{E}_H[\mathbf{x}(y-\mathbf{x}'\mathbf{\beta}_{LASSO}(H))]+2\lambda\sign(\mathbf{\beta}_{LASSO}(H))\right)_j=0\qquad j=1,\ldots,k
\end{align*}
Let $(\mathbf{x},y)\sim H_0$ denote the model distribution and $H_\epsilon$ the contaminated distribution. Then the FOC at the contaminated model is
\begin{align*}
-(1-\epsilon)\mathbb{E}_{H_0}[x_j(y-\mathbf{x}'\mathbf{\beta}_{LASSO}(H_\epsilon))]-\epsilon(\mathbf{x}_0)_j(y-\mathbf{x}_0'\mathbf{\beta}_{LASSO}(H_\epsilon))+\lambda\sign((\beta_{LASSO}(H_\epsilon))_j)=0.
\end{align*}
After differentiating with respect to $\epsilon$, we get
\begin{align*}
&\mathbb{E}_{H_0}\left[x_j(y-\mathbf{x}'\mathbf{\beta}_{LASSO}(H_\epsilon))\right] + (1-\epsilon)\left(\mathbb{E}_{H_0}[x_j\mathbf{x}']\frac{\partial\mathbf{\beta}_{LASSO}(H_\epsilon)}{\partial\epsilon}\right)-\\
&\qquad-(\mathbf{x}_0)_j\left(y-\mathbf{x}_0'\mathbf{\beta}_{LASSO}(H_\epsilon)\right)+\epsilon(\mathbf{x}_0)_j\left(\mathbf{x}_0'\frac{\partial\mathbf{\beta}_{LASSO}(H_\epsilon)}{\partial\epsilon}\right)=0.
\end{align*}
Taking the limit as $\epsilon$ approaches 0 gives an implicit definition of the influence function for $j=1,\ldots,k$
\begin{align}
\label{eq:lassomulti1}
\mathbb{E}_{H_0}[x_j\mathbf{x}']\cdot IF((\mathbf{x}_0,y_0),&\mathbf{\beta}_{LASSO},H_0) =\\
&= (\mathbf{x}_0)_j(y-\mathbf{x}_0'\mathbf{\beta}_{LASSO}(H_0)) - \mathbb{E}_{H_0}[x_j(y-\mathbf{x}'\mathbf{\beta}_{LASSO}(H_0))].\notag
\end{align}
For variables $j=k+1,\ldots,p$ with $(\mathbf{\beta}_{LASSO})_j=0$, we need to use subgradients \citep[][]{Bertsekas} to get the FOC
\begin{align*}
\mathbf{0}\,\in\,-\mathbb{E}_H[\mathbf{x}(y-\mathbf{x}'\mathbf{\beta}_{LASSO}(H))]+\lambda\cdot\partial\left(\|\mathbf{\beta}_{LASSO}(H)\|_1\right).
\end{align*}
Observing each variable individually yields
\begin{align}
|\mathbb{E}_H\left[x_j(y-\mathbf{x}'\mathbf{\beta}_{LASSO}(H))\right]|\leq\lambda.
\label{eq:FOCbeta0}
\end{align}
The coordinate descent algorithm converges for any starting value $\mathbf{\beta^\ast}$ to $\mathbf{\beta}_{LASSO}$ \citep[][]{Friedman, Tseng}, i.e. after enough updates $\mathbf{\beta^\ast}\approx \mathbf{\beta}_{LASSO}$. Thus, for $(\mathbf{\beta}_{LASSO}(H_0))_j=0$ and $(\mathbf{x},y)\sim H_0$, Equation (\ref{eq:FOCbeta0}) yields
\begin{align*}
\left|\mathbb{E}_{H_0}[x_j(y-{\mathbf{x}^{(j)}}'\mathbf{\beta^\ast}^{(j)})]\right| \leq \lambda.
\end{align*}
Lemma \ref{lemma:lasso_shooting} tells us then that
\begin{align*}
IF((\mathbf{x}_0,y_0), (\mathbf{\beta}_{LASSO})_j, H_0) = 0 \qquad \forall \, j=k+1,\ldots,p.
\end{align*}
With this we can rewrite Equation (\ref{eq:lassomulti1}) as
\begin{align*}
\mathbb{E}_{H_0}[\mathbf{x}_{1:k}\mathbf{x}'_{1:k}]\cdot IF((\mathbf{x}_0,y_0),&(\mathbf{\beta}_{LASSO})_{1:k},H_0) = \\
&= (\mathbf{x}_0)_{1:k}(y-\mathbf{x}_0'\mathbf{\beta}_{LASSO}(H_0)) - \mathbb{E}_{H_0}[x_{1:k}(y-\mathbf{x}'\mathbf{\beta}_{LASSO}(H_0))].
\end{align*}
Multiplying with $\mathbb{E}_{H_0}[\mathbf{x}_{1:k}\mathbf{x}'_{1:k}]^{-1}$ from the left side, we get the influence function of the lasso functional (\ref{eq:IF_lasso_multi}).
\end{proof}

\begin{proof}[Proof of Lemma \ref{lemma:approx_tan}]
We apply Proposition \ref{prop:gen} with a quadratic loss function and use the second derivative of the penalty function $J_K$
\begin{align*}
J''_K((\mathbf{\beta}_K)_j)=\begin{cases}
J''_K((\mathbf{\beta}_K)_j) =: a_j & j=1,\ldots,k\\
2K & j=k+1,\ldots,p.
\end{cases}
\end{align*}
W.l.o.g. we take $\sigma = 1$. This gives the influence function of $\mathbf{\beta}_K(H_0)$ 
\begin{align*}
IF((\mathbf{x}_0,y_0),\mathbf{\beta}_K, H_0) = \, & (\mathbb{E}_{H_0}[\mathbf{xx}'] + \lambda \diag(J''_K((\mathbf{\beta}_K)_1),\ldots,J''_K((\mathbf{\beta}_K)_k), 2K,\ldots,2K))^{-1} \cdot \\
& \cdot ((y_0-\mathbf{x}_0'\mathbf{\beta}_K(H_0))\mathbf{x}_0 - \mathbb{E}_{H_0}[(y-\mathbf{x}'\mathbf{\beta}_K(H_0))\mathbf{x}])
\end{align*}
The covariance matrix $\mathbb{E}_{H_0}[\mathbf{xx}']$ can be denoted as a block matrix
\begin{align*}
\mathbb{E}_{H_0}[\mathbf{xx}'] = \begin{pmatrix}
E_{11} & E_{12}\\
E_{21} & E_{22}
\end{pmatrix}.
\end{align*}
The inverse matrix needed in the influence function is then
\begin{align}
(\mathbb{E}_{H_0}[\mathbf{xx}'] + \lambda \diag(&J''_K((\mathbf{\beta}_K)_1),\ldots,J''_K((\mathbf{\beta}_K)_k), 2K,\ldots,2K))^{-1} = \notag\\
&=\begin{pmatrix}
E_{11}+\lambda \diag(J''_K((\mathbf{\beta}_K)_{1:k})) & E_{12}\\
E_{21} & E_{22} + 2\lambda K I_{p-k}
\end{pmatrix} ^{-1}.
\label{eq:inv}
\end{align}
The inverse of the block matrix can be computed as
\begin{align*}
(\mathbb{E}_{H_0}[\mathbf{xx}'] + \lambda \diag(0,\ldots,0, 2K,\ldots,2K))^{-1} = \begin{pmatrix}
A^{-1} + A E_{12}C^{-1}E_{21}A^{-1} & -A^{-1}E_{12}C^{-1}\\
-C^{-1}E_{21}A^{-1} & C^{-1}
\end{pmatrix}
\end{align*}
with $C = E_{22}+2\lambda K I_{p-k} - E_{21}A^{-1}E_{12}$ and $A=E_{11} + \lambda\diag(J''_K((\mathbf{\beta}_K)_{1:k}))$ \citep[see][p11]{Magnus}.

We denote the eigenvalues of matrix $D=E_{22} - E_{21}E_{11}^{-1}E_{12}$ by $\nu_1,\ldots,\nu_{p-k}$. Then the eigenvalues of the symmetric positive definite matrix $C$ are $\nu_1+2\lambda K,\ldots,\nu_{p-k}+2\lambda K$. If $K$ approaches infinity, these eigenvalues also tend to infinity. Hence, all eigenvalues of $C^{-1}$ converge to zero. Thus, $C^{-1}$ becomes the zero matrix and therefore the inverse matrix in (\ref{eq:inv}) converges to
\begin{align*}
\lim_{K\rightarrow\infty} (\mathbb{E}_{H_0}[\mathbf{xx}'] + \lambda \diag(0,\ldots,0, 2K,\ldots,2K)) ^{-1} = \begin{pmatrix}
E_{11}^{-1} & \mathbf{0}\\
\mathbf{0} & \mathbf{0}
\end{pmatrix}.
\end{align*}
This gives the influence function of the lasso functional (\ref{eq:IF_lasso_multi}) as the limit of $IF((\mathbf{x}_0, y_0),\mathbf{\beta}_K,H_0)$ for $K\rightarrow\infty$.
\end{proof}

\begin{proof}[Proof of Lemma \ref{lemma:IFspLTS}]
As the sparse LTS functional is continuous, the influence function of the sparse LTS functional equals $0$ if $\beta_{spLTS}(H_0)=0$. Thus, assume from now on $\beta_{spLTS}(H_0)\neq0$.

The first-order condition at the contaminated model $H_\epsilon = (1-\epsilon)H_0+\epsilon\delta_{(x_0,y_0)}$ yields
\begin{align}
0=\frac{\partial}{\partial\beta}\left(\int_{-q_{\epsilon,\beta}}^{q_{\epsilon,\beta}}u^2dH_\epsilon^\beta(u)\right)+\alpha\lambda\sign(\beta)=:\Psi(\epsilon,\beta).
\label{eq:aux}
\end{align}
Note that here the quantile $q_{\epsilon,\beta}$ as well as the joint model distribution $H_\epsilon^\beta$ of $x$ and $y$  depend on $\beta$. We denote the solution of (\ref{eq:aux}) by $\beta_\epsilon:=\beta_{spLTS}(H_\epsilon)$ for $\epsilon\neq0$ and $\beta_{spLTS}(H_0)$ otherwise.

As (\ref{eq:aux}) is true for all $\epsilon\in\mathbb{R}_+$, the chain rule gives
\begin{gather}
0 = \frac{\partial}{\partial\epsilon}[\Psi(\epsilon,\beta_\epsilon)]|_{\epsilon = 0} = \Psi_1(0,\beta_{spLTS}(H_0)) + \Psi_2(0,\beta_{spLTS}(H_0)) IF(\beta_{spLTS})\notag\\
\leadsto IF(\beta_{spLTS}) = -[\Psi_2(0,\beta_{spLTS}(H_0))]^{-1}\Psi_1(0,\beta_{spLTS}(H_0))
\label{eq:IFpsi}
\end{gather}
where $\Psi_1(a,b) = \frac{\partial}{\partial \epsilon} \Psi(\epsilon,b)|_{\epsilon=a}$ and $\Psi_2(a,b) = \frac{\partial}{\partial \beta} \Psi(a,\beta)|_{\beta=b}$.

Before computing $\Psi_1(0,\beta_{spLTS}(H_0))$ and $\Psi_2(0,\beta_{spLTS}(H_0))$, we can simplify $\Psi(\epsilon,\beta)$ by using $H_0^\beta = \mathcal{N}(0,\sigma^2(\beta))$ with $\sigma^2(\beta)=\sigma^2 +(\beta_{spLTS}(H_0) - \beta)^2 \Sigma$, as $x\sim\mathcal{N}(0,\Sigma)$ and $e\sim\mathcal{N}(0,\sigma^2)$
\begin{align*}
\Psi(\epsilon,\beta) &= \frac{\partial}{\partial\beta}\left((1-\epsilon)\int_{-q_{\epsilon,\beta}}^{q_{\epsilon,\beta}} u^2 dH_0^\beta(u) + \epsilon I_{[|y_0-x_0\beta|\leq q_{\epsilon,\beta}]} (y_0-x_0\beta)^2\right)+\alpha\lambda\sign(\beta)\\
&=(1-\epsilon)\frac{\partial}{\partial\beta}\left(\int_{-q_{\epsilon,\beta}}^{q_{\epsilon,\beta}} \frac{u^2}{\sigma(\beta)}\phi(\frac{u}{\sigma(\beta)})du\right) - 2\epsilon x_0(y_0-x_0\beta)I_{[|y_0-x_0\beta|\leq q_{\epsilon,\beta}]}+\alpha\lambda\sign(\beta)\\
\end{align*}
and the Leibniz integral rule
\begin{align*}
\frac{\partial}{\partial\beta}\left(\int_{-q_{\epsilon,\beta}}^{q_{\epsilon,\beta}} \frac{u^2}{\sigma(\beta)}\phi(\frac{u}{\sigma(\beta)})du\right) = \int_{-q_{\epsilon,\beta}}^{q_{\epsilon,\beta}} u^2\phi(\frac{u}{\sigma(\beta)})(1-\frac{u^2}{\sigma^2(\beta)}&)du\frac{(\beta_0-\beta)\Sigma}{\sigma^3(\beta)}+\\
&+2\frac{q_{\epsilon,\beta}^2}{\sigma(\beta)}\phi(\frac{q_{\epsilon,\beta}}{\sigma(\beta)})\frac{\partial}{\partial\beta}[q_{\epsilon,\beta}].
\end{align*}

To obtain the derivative $\Psi_1(0,\beta_{spLTS}(H_0))$, we can again use the Leibniz integral rule
\begin{align*}
\Psi_1(&0,\beta_{spLTS}(H_0)) =\\
& -\bigg(\int_{-q_{0,\beta_{spLTS}(H_0)}}^{q_{0,\beta_{spLTS}(H_0)}} u^2 \phi(\frac{u}{\sigma(\beta_{spLTS}(H_0))})(1-\frac{u^2}{\sigma^2(\beta_{spLTS}(H_0))})du\frac{(\beta_0-\beta_{spLTS}(H_0))\Sigma}{\sigma^3(\beta_{spLTS}(H_0))} +\\
&\hspace{4cm}+2\frac{q_{0,\beta_{spLTS}(H_0)}^2}{\sigma(\beta_{spLTS}(H_0))}\phi(\frac{q_{0,\beta_{spLTS}(H_0)}}{\sigma(\beta_{spLTS}(H_0))})\frac{\partial}{\partial\beta} [q_{0,\beta}]|_{\beta=\beta_{spLTS}(H_0)}\bigg) +\\
&+\frac{\partial}{\partial\epsilon}\bigg[\int_{-q_{\epsilon,\beta_{spLTS}(H_0)}}^{q_{\epsilon,\beta_{spLTS}(H_0)}} u^2 \phi(\frac{u}{\sigma(\beta_{spLTS}(H_0))})(1-\frac{u^2}{\sigma^2(\beta_{spLTS}(H_0))})du\bigg]\bigg|_{\epsilon = 0}\frac{(\beta_0-\beta_{spLTS}(H_0))\Sigma}{\sigma^3(\beta_{spLTS}(H_0))} +\\
&+4\frac{q_{0,\beta_{spLTS}(H_0)}}{\sigma(\beta_{spLTS}(H_0))}\frac{\partial}{\partial\epsilon}[q_{\epsilon,\beta_{spLTS}(H_0)}]|_{\epsilon = 0}\,\phi(\frac{q_{0,\beta_{spLTS}(H_0)}}{\sigma(\beta_{spLTS}(H_0))})\frac{\partial}{\partial\beta}[q_{0,\beta}]|_{\beta=\beta_{spLTS}(H_0)}+\\
&+2\frac{q_{0,\beta_{spLTS}(H_0)}^2}{\sigma(\beta_{spLTS}(H_0))}\phi'(\frac{q_{0,\beta_{spLTS}(H_0)}}{\sigma(\beta_{spLTS}(H_0))})\frac{\partial}{\partial\epsilon}[q_{\epsilon,\beta_{spLTS}(H_0)}]|_{\epsilon=0} \frac{1}{\sigma(\beta_{spLTS}(H_0))}\frac{\partial}{\partial\beta}[q_{0,\beta}]|_{\beta=\beta_{spLTS}(H_0)}+\\
&+2\frac{q_{0,\beta_{spLTS}(H_0)}^2}{\sigma(\beta_{spLTS}(H_0))}\phi(\frac{q_{0,\beta_{spLTS}(H_0)}}{\sigma(\beta_{spLTS}(H_0))})\frac{\partial}{\partial\epsilon}[\frac{\partial}{\partial\beta}[q_{\epsilon,\beta}]|_{\beta=\beta_{spLTS}(H_0)}]|_{\epsilon=0} -\\
&-2x_0(y_0-x_0\beta_{spLTS}(H_0))I_{[|y_0-x_0\beta_{spLTS}(H_0)|\leq q_{0,\beta_{spLTS}(H_0)}]}.
\end{align*}

To compute the derivatives of the quantiles, we denote the distribution of $|y-\mathbf{x}'\mathbf{\beta}|$ by $\bar{H}_\epsilon^{\mathbf{\beta}}$ when $(\mathbf{x},y)\sim H_\epsilon$. Using the equations $\bar{H}_\epsilon^\beta(q_\epsilon,\beta)=\alpha$ and $\bar{H_0}^\beta(q_0,\beta)=\alpha$ and differentiating w.r.t. the required variables yields
\begin{gather*}
\frac{\partial}{\partial\epsilon}[q_{\epsilon,\beta_{spLTS}(H_0)}]|_{\epsilon=0}=\frac{\alpha - I_{[|y_0-x_0\beta_{spLTS}(H_0)|\leq q_{0,\beta_{spLTS}(H_0)}]}}{2\phi(q_\alpha)\frac{1}{\sigma(\beta_{spLTS}(H_0))}}\\
\frac{\partial}{\partial\beta}[q_{0,\beta}]|_{\beta=\beta_{spLTS}(H_0)}=-\frac{q_{0,\beta_{spLTS}(H_0)}(\beta_0-\beta_{spLTS}(H_0))\Sigma}{\sigma^2(\beta_{spLTS}(H_0))}\\
\frac{\partial}{\partial\epsilon}[\frac{\partial}{\partial\beta}[q_{\epsilon,\beta}]|_{\beta=\beta_{spLTS}(H_0)}]|_{\epsilon=0}=\frac{I_{[|r_0|\leq q_\alpha]}-\alpha}{2\phi(q_\alpha)}\cdot\frac{(\beta_0-\beta_{spLTS}(H_0))\Sigma}{\sigma(\beta_{spLTS}(H_0))}
\end{gather*}
with $r_0:=\frac{y_0-x_0\beta_{spLTS}(H_0)}{\sigma(\beta_{spLTS}(H_0))}$.

Thus,
\begin{align}
\Psi_1(0,\beta_{spLTS}(H_0)) =&(-4q_\alpha\phi(q_\alpha)+2\alpha+2q_\alpha^2(I_{[|r_0|\leq q_\alpha]}-\alpha))(\beta_0-\beta_{spLTS}(H_0))\Sigma\\
&-2x_0(y_0-x_0\beta_{spLTS}(H_0))I_{[|r_0|\leq q_\alpha]}.
\label{eq:psi1}
\end{align}

With similar ideas as in the derivation of $\Psi_1(0,\beta_{spLTS}(H_0))$, we get
\begin{align}
\Psi_2(0,\beta_{spLTS}(H_0)) =(-4q_\alpha\phi(q_\alpha)+4\Phi(q_\alpha)-2)\Sigma.
\label{eq:psi2}
\end{align}

Using (\ref{eq:psi1}) and (\ref{eq:psi2}) in (\ref{eq:IFpsi}), we get the influence function (\ref{eq:IFspLTS}) of the sparse LTS functional for simple regression.
\end{proof}

\bibliographystyle{plainnat}
\bibliography{IF_Revision2}

\end{document}